\documentclass[11pt,reqno,letterpaper]{amsart}
\usepackage{color}
\usepackage[colorlinks=true, allcolors=blue,backref=page]{hyperref}
\usepackage{amsmath, amssymb, amsthm}
\usepackage{mathrsfs}
\usepackage{mathtools}
\usepackage[noabbrev,capitalize,nameinlink]{cleveref}
\crefname{equation}{}{}
\usepackage{fullpage}
\usepackage[noadjust]{cite}
\usepackage{graphics}
\usepackage{pifont}
\usepackage{tikz}
\usepackage{bbm}
\usepackage[T1]{fontenc}

\usetikzlibrary{arrows.meta}

\usepackage{environ}
\usepackage{framed}
\usepackage{url}
\usepackage[linesnumbered,ruled,vlined]{algorithm2e}
\usepackage[noend]{algpseudocode}
\usepackage[labelfont=bf]{caption}
\usepackage{cite}
\usepackage{framed}
\usepackage[framemethod=tikz]{mdframed}
\usepackage{appendix}
\usepackage{graphicx}
\usepackage[textsize=tiny]{todonotes}
\usepackage{tcolorbox}
\usepackage{enumerate}
\allowdisplaybreaks[1]
\usepackage{enumerate}
\usepackage{stmaryrd}
\usepackage[margin=1in]{geometry}

\usepackage[shortlabels]{enumitem}
\crefformat{enumi}{#2#1#3}
\crefrangeformat{enumi}{#3#1#4 to~#5#2#6}
\crefmultiformat{enumi}{#2#1#3}
{ and~#2#1#3}{, #2#1#3}{ and~#2#1#3}

\DeclareSymbolFont{symbolsC}{U}{pxsyc}{m}{n}
\SetSymbolFont{symbolsC}{bold}{U}{pxsyc}{bx}{n}
\DeclareFontSubstitution{U}{pxsyc}{m}{n}
\DeclareMathSymbol{\medcircle}{\mathbin}{symbolsC}{7}

\crefname{algocf}{Algorithm}{Algorithms}

\crefname{equation}{}{} 
\AtBeginEnvironment{appendices}{\crefalias{section}{appendix}} 

\usepackage[color,final]{showkeys} 

\colorlet{refkey}{orange!20}
\colorlet{labelkey}{blue!30}

\crefname{algocf}{Algorithm}{Algorithms}

\numberwithin{equation}{section}
\newtheorem{theorem}{Theorem}[section]
\newtheorem{proposition}[theorem]{Proposition}
\newtheorem{lemma}[theorem]{Lemma}
\newtheorem{claim}[theorem]{Claim}

\crefname{claim}{Claim}{Claims}

\newtheorem{corollary}[theorem]{Corollary}
\newtheorem{conjecture}[theorem]{Conjecture}
\newtheorem*{question*}{Question}

\theoremstyle{definition}
\newtheorem{definition}[theorem]{Definition}

\newtheorem*{definition*}{Definition}

\theoremstyle{remark}
\newtheorem*{remark}{Remark}

\newcommand{\mb}{\mathbb}

\newcommand{\mc}{\mathcal}
\newcommand{\mf}{\mathfrak}
\newcommand{\mr}{\mathrm}

\newcommand{\on}{\operatorname}

\newcommand{\eps}{\varepsilon}

\let\originalleft\left
\let\originalright\right
\renewcommand{\left}{\mathopen{}\mathclose\bgroup\originalleft}
\renewcommand{\right}{\aftergroup\egroup\originalright}

\allowdisplaybreaks

\newif\ifpublic
\publictrue

\ifpublic

\newcommand{\ignore}[1]{}

\else

\fi

\title{Threshold for Steiner triple systems}

\author[A2]{Ashwin Sah}
\author[A3]{Mehtaab Sawhney}
\address{Department of Mathematics, Massachusetts Institute of Technology, Cambridge, MA 02139, USA}
\email{\{asah,msawhney\}@mit.edu}

\author[Simkin]{Michael Simkin}
\address{Center of Mathematical Sciences and Applications, Harvard University, Cambridge, MA 02138, USA}
\email{msimkin@cmsa.fas.harvard.edu}

\thanks{Sah and Sawhney were supported by NSF Graduate Research Fellowship Program DGE-1745302. Sah was supported by the PD Soros Fellowship. Simkin was supported by the Center of Mathematical Sciences and Applications at Harvard University.}

\begin{document}

\maketitle
\begin{abstract}
We prove that with high probability $\mb{G}^{(3)}(n,n^{-1+o(1)})$ contains a spanning Steiner triple system for $n\equiv 1,3\pmod{6}$, establishing the exponent for the threshold probability for existence of a Steiner triple system. We also prove the analogous theorem for Latin squares. Our result follows from a novel bootstrapping scheme that utilizes iterative absorption as well as the connection between thresholds and fractional expectation-thresholds established by Frankston, Kahn, Narayanan, and Park.
\end{abstract}

\section{Introduction}\label{sec:introduction}

A foundational result in the theory of random graphs, due to Erd\H{o}s and R\'enyi \cite{ER66}, is that the threshold for the appearance of perfect matchings in $\mb{G}(n,p)$ is $\log n / n$. It is natural to seek higher-dimensional analogues of this result. As the simplest case, consider perfect matchings in $3$-uniform hypergraphs. These are collections of $3$-edges, or \textit{triangles}, such that each vertex is contained in exactly one triangle. Let $\mb{G}^{(3)}(n,p)$ denote the binomial random $n$-vertex hypergraph in which each triangle is present with probability $p$. Determining the threshold for the appearance of perfect matchings in this model is the well-known ``Shamir's problem'' \cite{SS83}, which was resolved (up to a constant factor) in a seminal paper by Johansson, Kahn, and Vu \cite{JKV08} with sharp threshold and hitting time results obtained in later work of Kahn \cite{Kahn19,Kahn22}.

Of course, there is more than one high-dimensional analogue to (graphical) perfect matchings. It is just as natural to consider (spanning) \textit{Steiner triple systems} (i.e., triangle sets in which each \textit{pair} of vertices is contained in exactly one triangle) and their appearance in $\mb{G}^{(3)}(n,p)$. Here, much less is known and until recently such results seemed out of reach. Indeed, even the log-asymptotics of the number of Steiner triple systems was a mystery until the work of Keevash \cite{Kee18} (which built on his earlier breakthrough establishing the existence of designs \cite{Kee14}). Furthermore, both of these problems can be viewed under the common umbrella of determining the threshold for the existence of designs relative to a random set, with Shamir's problem corresponding to $(r,1)$-designs and Steiner triple systems to $(3,2)$-designs.

Our main result is the determination of the threshold for Steiner triple systems up to a sub-polynomial factor, which is the first higher-dimensional generalization of Erd\H{o}s--R\'enyi and Johansson--Kahn--Vu incorporating nontrivial designs.

\begin{theorem}\label{thm:main}
	Let $n \in \mb N$ satisfy $n \equiv 1,3 \pmod 6$. Let $\mc H \sim \mb{G}^{(3)} \left(n,\exp(C(\log n)^{3/4})/n \right)$, with $C>0$ a sufficiently large constant. With high probability\footnote{We say that a sequence of events, parameterized by $n$, holds \textit{with high probability} (\textit{w.h.p.}) if the probabilities of their occurrence tend to $1$.}, $\mc H$ contains an order-$n$ Steiner triple system.
\end{theorem}

\begin{remark}
	An order-$n$ Steiner triple system is equivalent to a triangle-decomposition of the complete graph $K_n$. A graph has a triangle-decomposition only if it is \textit{triangle-divisible}, i.e., its every degree is even and the number of edges is a multiple of $3$. For $K_n$ this is equivalent to the arithmetic condition $n \equiv 1,3 \pmod 6$, demonstrating the necessity of this assumption. That Steiner triple systems indeed exist whenever $n$ satisfies this condition is a famous classical theorem of Kirkman~\cite{kirkman1847problem}.
\end{remark}

Note that $\exp((\log n)^{3/4}) = n^{o(1)}$. Thus, \cref{thm:main} implies that the threshold for the appearance of Steiner triple systems is bounded above by $n^{-1+o(1)}$. A corresponding lower bound is obtained by observing that if $\mc{H}$ contains a Steiner triple system then every edge of $K_n$ is contained in at least one triangle of $\mc{H}$. A straightforward calculation (analogous to that for isolated vertices in $\mb{G}(n,p)$) reveals that the (sharp) threshold for this modified property is $2 \log n / n = n^{-1+o(1)}$. Hence, our result establishes the threshold up to a subpolynomial factor.

A recent breakthrough relating thresholds and fractional expectation-thresholds by Frankston, Kahn, Narayanan, and Park \cite{FKNP21} (and also a very recent breakthrough of Park and Pham \cite{PP22} resolving the Kahn--Kalai conjecture), as a corollary, gave an alternate and substantially simpler proof of the result of Johansson, Kahn, and Vu. This proof hinges on the ability to determine the \textit{fractional expectation-threshold}, which can be viewed as a linear program whose variables are all sets of hyperedges.

The sheer size of such programs suggests that determining the expectation-threshold or fractional expectation threshold is a difficult task in general. Instead, in previous applications, the uniform distribution on some desired class of objects is used to witness a lower bound, via a parameter known as ``spread''. In this light, the application of \cite{FKNP21} to Shamir's problem relies crucially on the fact that the enumeration of $(r,1)$-designs (as well as extensions of partial $(r,1)$-designs) is straightforward. However, determining the fractional expectation-threshold of Steiner triple systems is nontrivial; as noted in \cite[Section~8.D]{FKNP21} the error terms in the enumerative results of Keevash \cite{Kee18} are too large to prove that the uniform distribution on Steiner triple systems has sufficiently small spread. Therefore the authors of \cite{FKNP21} raise applying these methods to combinatorial designs as an interesting open problem. In this paper, we circumvent this difficulty by \textit{constructing} a (non-uniform) distribution on Steiner triple systems with small spread. We expect this approach to have further applications.

Steiner triple systems are closely related to \textit{Latin squares} (i.e., $n \times n$ matrices in which every row and column is a permutation of $\{1,2,\ldots,n\}$). The latter are naturally equivalent to (labeled) triangle-decompositions of $K_{n,n,n}$, with the three vertex parts corresponding to rows, columns, and symbols. With a few adjustments the proof of \cref{thm:main} yields a similar threshold result for Latin squares.

We use the following terminology: Let $S: [n]^2 \to 2^{[n]}$ be a function that assigns, to each cell in an $n \times n$ grid, a set of symbols from $\{1,2,\ldots,n\}$. Say that $S$ \textit{supports} an order-$n$ Latin square $L$ if for every $i,j \in [n]$ there holds $L(i,j) \in S(i,j)$. For $p \in [0,1]$ let $\mb{M}(n,p)$ be the distribution on functions $S: [n]^2 \to 2^{[n]}$ where for every $i,j,k \in [n]$, the symbol $k$ is included in $S(i,j)$ with probability $p$ independent of all other choices.

\begin{theorem}\label{thm:latin}
Let $n \in \mb{N}$ and let $\mc{S} \sim \mb{M}(n,\exp(C(\log n)^{3/4})/n)$, with $C>0$ a sufficiently large constant. W.h.p.\ $\mc{S}$ supports a Latin square.
\end{theorem}

Prior to this work the only known upper bounds on the thresholds in \cref{thm:main,thm:latin} were quite far from the $n^{o(1)-1}$ proved here. For Latin squares, Andr\'en, Casselgren, and \"Ohman~\cite{ACO13} proved that there exists a constant $p < 1$ such that w.h.p.\ $\mb{M}(n,p)$ supports a Latin square. For Steiner triple systems, Simkin \cite{Sim17} observed that Keevash's method of randomized algebraic construction~\cite{Kee14} can be used to show that for a sufficiently small $\varepsilon > 0$, w.h.p.\ $\mb{G}^{(3)}(n,n^{-\varepsilon})$ contains a Steiner triple system. 

\newcommand{\psts}{{p_{\mr{STS}}}}

Before moving on to proofs we mention an interesting consequence of \cref{thm:main}: the threshold for the appearance of Steiner triple systems is sharp, in the following sense.

\begin{corollary}\label{cor:sharp-threshold}
There is a function $\psts(n)$ such that for all $\varepsilon > 0$, when $n\equiv 1,3\pmod{6}$ w.h.p.~$\mb{G}^{(3)}(n,(1+\varepsilon)\psts(n))$ contains a Steiner triple system but w.h.p.~$\mb{G}^{(3)}(n,(1-\varepsilon)\psts(n))$ does not.
\end{corollary}

This is surprising, since we have not determined what the threshold actually is. Nevertheless, \cref{cor:sharp-threshold} follows from Friedgut's characterization of sharp thresholds~\cite{Fri99,Fri05}. Indeed, for $n\equiv 1,3 \pmod 6$, let $\psts(n)$ be the threshold for containing Steiner triple systems (i.e.,~$\mb{G}^{(3)}(n,\psts(n))$ contains a Steiner triple system with probability $1/2$). \cref{thm:main} tells us that $\psts(n) \leq n^{o(1)-1}$. On the other hand, by considering the disappearance of vertex pairs not contained in a triangle, we concluded that $\psts(n) = \Omega(n^{-1}\log n)$. Hence, $\psts(n) \neq \Theta(n^\alpha)$ for any $\alpha \in \mb R$. However, a consequence of \cite[Theorem 2.1]{Fri05} and the remarks immediately after is that in our setting (sampling hypergraphs), coarse thresholds are limited to the form $\Theta(n^\alpha)$, implying that our threshold is sharp.

Regarding Latin squares, the threshold is sharp for essentially the same reasons. Although triangle-decompositions of $K_{n,n,n}$ are not invariant under vertex permutations so that \cite{Fri05} does not directly apply, similar methods can be used to deduce a sharp threshold \cite{Fri22}.

\subsection{Techniques for bounding thresholds}

Finding thresholds for spanning structures in random graphs and hypergraphs has played a major role in the field since its inception. Prominent examples include thresholds for containing a spanning tree (which is equivalent to connectivity)~\cite{ER59}, a perfect matching \cite{ER64,ER66}, a Hamilton cycle \cite{Pos76}, a triangle-factor \cite{JKV08}, and a given bounded-degree spanning tree \cite{Mon19}.

Lower bounds on the thresholds for each of these properties (and many others) can be obtained in more or less the same way: fixing a vertex $v$, it is contained in a spanning tree or perfect matching only if its degree is at least $1$. Similarly, it is contained in a Hamilton cycle only if its degree is at least $2$. Finally, it is contained in a triangle-factor only if it is contained in at least one triangle. By computing the expectation of each of these random variables and applying Markov's inequality we obtain a lower bound of $\Omega(n^{-1})$ for connectivity, perfect matchings, and Hamiltonicity, and a lower bound of $\Omega(n^{-2/3})$ for existence of a triangle-factor. Avoiding formal definitions (which can be found in \cite{KK07}), the maximal lower bound obtained by similar arguments is known as the \textit{expectation-threshold} for the property.

Surprisingly, these easily-obtained lower bounds turn out to be within a logarithmic factor of the true thresholds. However, in sharp contrast to the lower bounds, the original proofs of the corresponding upper bounds are problem-specific. This disparity (and the associated difficulty of obtaining thresholds for some properties, as in Shamir's problem) motivated a family of beautiful conjectures of Kahn and Kalai \cite{KK07}. The main conjecture is that the threshold for a monotone property is always within a logarithmic factor of its expectation-threshold.

In a very recent breakthrough, Park and Pham \cite{PP22} gave an ingenious proof of the Kahn--Kalai conjecture. However, for our application (and many others), a fractional version of the conjecture, due to Talagrand \cite{Tal10}, suffices. The so-called \textit{fractional expectation-threshold vs.\ threshold conjecture} was proved by Frankston, Kahn, Narayanan, and Park \cite{FKNP21} in an earlier breakthrough. These works are related to yet another, yet earlier, breakthrough: the advance on the \textit{sunflower conjecture} due to Alweiss, Lovett, Wu, and Zhang \cite{ALWZ21}. For our purposes it suffices to consider a corollary of these results, for which we need the next definition.

\begin{definition}
Consider a finite ground set $Z$ and fix a nonempty collection of subsets $\mc{H} \subseteq 2^Z$. Let $\mu$ be a probability measure on $\mc{H}$. For $q > 0$ we say that $\mu$ is \textit{$q$-spread} if for every set $S \subseteq Z$:
\[\mu \left( \{ A \in \mc{H} : S \subseteq A \} \right) \le q^{|S|}.\]
\end{definition}

The next theorem, relating spread measures and thresholds, is due to Frankston, Kahn, Narayanan, and Park \cite{FKNP21}. We have slightly tailored it to our setting.

\begin{theorem}[{From \cite[Theorem~1.6]{FKNP21}}]\label{thm:FKNP}
	There exists a constant $C = C_{\ref{thm:FKNP}} > 0$ such that the following holds. Consider a non-empty ground set $Z$ and fix a nonempty collection of subsets $\mc{H}\subseteq 2^Z$. Suppose that there exists a $q$-spread probability measure on $\mc{H}$.  Then a random binomial subset of $Z$ where each element is sampled with probability $\min(Cq\log|Z|,1)$ contains an element of $\mc{H}$ as a subset with probability at least $3/4$.
\end{theorem}

For many graph families, including those mentioned above, the spread of the uniform distribution is easily seen to match the lower bounds on the threshold. Thus, \cref{thm:FKNP} immediately determines these thresholds up to a logarithmic factor. However, \cref{thm:FKNP} does not immediately imply any bound at all on the threshold for $\mb{G}^{(3)}(n,p)$ to contain a Steiner triple system. The issue is that currently, our understanding of the uniform distribution on Steiner triple systems is rather poor. Indeed, as remarked in \cite[Section~8.D]{FKNP21}, even the uncertainty in the \textit{number} of order-$n$ Steiner triple systems is large enough that it precludes any useful bounds on the spread.

Although the uniform distribution is the most natural one with which to apply \cref{thm:FKNP}, this is certainly not required. We prove \cref{thm:main} by designing a distribution on Steiner triple systems that is $n^{o(1)-1}$-spread \textit{by construction}, and then applying \cref{thm:FKNP}.

\subsection{Spread distributions and iterative absorption}

The $n^{o(1)-1}$-spread distribution used to prove \cref{thm:main} is defined implicitly by a randomized algorithm to construct Steiner triple systems. In order to outline this algorithm we briefly recount some recent breakthroughs in design theory.

We begin with the \textit{triangle removal process}, which is closely related to the influential \textit{R\"odl nibble}~\cite{Rod85}. This is the following random greedy algorithm to construct a partial Steiner triple system: Beginning with $G=K_n$ repeatedly and for as long as possible delete, uniformly at random, a triangle from $G$ and add it to a growing collection of triples. Spencer \cite{Spe95} and R\"odl and Thoma \cite{RT96} independently proved that w.h.p.\ this process terminates when $G$ has only $o(n^2)$ edges. Equivalently, this method produces an approximate Steiner triple system.

It is straightforward to adapt the triangle removal process so that the triangle set it produces has spread $O(n^{o(1)-1})$. Perhaps the simplest way is to first restrict the available triangles to a prescribed binomial random subset of density (say) $(\log n)^2/n$, and then show that the process is still likely to construct an approximate Steiner triple system.

Given the success of the triangle removal process, a natural way to construct an exact Steiner triple system is to find a triangle-decomposition of the edges remaining at the end of the process\footnote{Strictly speaking, one must stop the triangle removal process before its natural termination, at which point there are no triangles in $G$ by definition.}. This is essentially what Keevash does with his breakthrough method of \textit{randomized algebraic constructions} \cite{Kee14,Kee18}. Moreover, Keevash's method is incredibly powerful in that it proves the existence of designs with arbitrary parameters, which was a central question in combinatorics since the nineteenth century. Unfortunately, the algebraic component of Keevash's construction has rather poor spread, and so is unsuitable for our application.

An alternative to Keevash's method is \textit{iterative absorption}, developed by K\"uhn, Osthus, and collaborators \cite{KO13,KKO15, BKLO16}. This method gave an alternate proof of the existence of designs \cite{GKLO16} using purely probabilistic and combinatorial methods. In this paper we mostly follow its specialization by Barber, Glock, K\"uhn, Lo, Montgomery, and Osthus \cite{BGKLMO20} to triangle-decompositions.

A key insight is that by using a modified version of the triangle removal process, the uncovered edges at the end of the process can be ``localized'' to a small vertex set $U_1 \subseteq V(K_n)$. That is, after fixing $U_1 \subseteq V(K_n)$ (satisfying, say, $|U_1| \approx \varepsilon n$ for a small $\varepsilon > 0$), a multi-stage randomized ``cover-down'' procedure can produce a partial Steiner triple system that covers all edges in $K_n$ not spanned by $U_1$. Furthermore, the graph of uncovered edges in $U_1$ is nearly-complete. By repeating this process the uncovered edges can be iteratively localized to sets $U_1 \supseteq U_2 \supseteq \cdots \supseteq U_\ell = X$, where $X$ may be quite small. Since the goal is to construct an exact Steiner triple system, the iterative process is preceded by setting aside an ``absorber'' for $X$. This is a graph $H \subseteq K_n$ with the property that for any possible remainder graph $L$ on $X$, the graph $H \cup L$ admits a triangle-decomposition.

Iterative absorption, as outlined in \cite{BGKLMO20}, does not itself produce a distribution with sufficient spread. The issue is that for each $U_i$, the triangle set constructed on $U_i$ forms a nearly complete graph and contains $\Omega(|U_i|^2)$ triangles. Thus, the best spread one can hope for in this method is $\Omega(|X|^{-1})$, which is far larger than $1/n$ since $X$ must be small in order to construct the absorber $H$.

As a remedy, our algorithm combines iterative absorption with a bootstrapping scheme that iteratively constructs distributions with better and better spread. Concretely, let $P(\eta)$ be the proposition that for every sufficiently large, near-complete, and triangle-divisible graph $G$ there exists an $n^{\eta}/n$-spread distribution over triangle-decompositions of $G$. Note that $P(0)$ would imply \cref{thm:main}, since we may take $G=K_n$. We remark that here, and in the remainder of this outline, our goal is to provide a clear and concise summary of our argument. Thus, we take some leeway and are not as precise with some statements as we will be in the proof. For example, proposition $P(\eta)$ is slightly different from its analogue, \cref{thm:boosting}$(\eta)$.

We define the sequence $\eta_k = 2/(2+k)$, and we will inductively show that $P(\eta_k)$ holds. Since $\eta_k \to 0$, this implies that there exists an $n^{o(1)-1}$-spread distribution of Steiner triple systems.  The fact that $P(\eta_0 = 1)$ holds is itself a non-trivial fact; this follows implicitly from \cite{Kee18} and explicitly from \cite{BGKLMO20}.

Now suppose that $P(\eta_{k-1})$ holds. We wish to show that $P(\eta_k)$ holds as well. Let $G$ be a large, near-complete, triangle-divisible graph. We wish to construct an $n^{\eta_k-1}$-spread distribution on the triangle-decompositions of $G$. We proceed as follows: We first set aside a vertex set $X \subseteq V(G)$ of size approximately $n^{1-\eta_k}$. Next, we set aside a small, random, absorbing triangle set $\mc H$ in $G$. The triangle set $\mc H$ resembles a binomial random triangle set of density $1/n$, ensuring that it does not negatively impact the spread. As we will explain momentarily, this absorber serves a different purpose than the absorber in \cite{BGKLMO20}.

After setting aside the absorber, we use iterative absorption to find a triangle set $\mc S$ that covers all edges in $G \setminus (E(\mc H) \cup G[X])$. Let $L \subseteq G[X]$ denote the graph of uncovered edges. We remark that since $\mc S$ is constructed by processes resembling the triangle removal process, it is straightforward to ensure that the spread of $\mc S$ is approximately $|X|^{-1} \approx n^{\eta_k-1}$.

Now, by assumption, there exists an $|X|^{\eta_{k-1}-1}$-spread distribution over the triangle-decompositions of $L$. Let $\mc L$ be a triangle-decomposition sampled according to this distribution. Observe that $\mc S \cup \mc L$ is a triangle-decomposition of $G$. However, due to the presence of $\mc L$, its spread is $|X|^{\eta_{k-1}-1} \gg n^{\eta_k-1}$. This is where the absorber comes in: its purpose is to ``spread'' the probability mass of the triangles induced by $X$ over the rest of the graph. This reduces the overall spread of the Steiner triple system. Specifically, $\mc H$ has the property that for every triangle $T$ in $X$, there exist many configurations $\mc H' \subseteq \mc H$ such that $\mc H' \cup \{T\}$ can be replaced by a triangle set that does not use $T$. Furthermore, given the triangle-decomposition $\mc L$ of $L$, it is possible to choose a set of such configurations $\{\mc H'_T\}_{T \in \mc L}$ that are mutually disjoint. Thus, all the triangles used in $\mc L$ can be replaced by a set of triangles that are not spanned by $X$. Finally, if the choice of these configurations are randomly chosen in an appropriate way, the spread of the resulting Steiner triple system is less than $n^{\eta_k-1}$. Since there exists an $n^{\eta_k-1}$-spread distribution of triangle-decompositions of $G$, the proposition $P(\eta_k)$ holds.
 
Finally, the modification for Latin squares is straightforward; we detail the minor changes in \cref{sec:latin-squares}.

\subsection{Absorbers as spread boosters}

We stress that unlike traditional uses of absorbers, we cannot intentionally plant specific absorbers for each triangle that might appear in $\mc{L}$ (though it may seem like there is available space within $K_n$ for such constructions). The reason is that any specific finite absorber w.h.p.\ will not appear at all in $\mb{G}(n,n^{\theta-1})$ for $\theta>0$ sufficiently small. In fact, as far as algorithmically finding absorbers in $\mc H$ goes, there seems to be a ``barrier'' around density $1/\sqrt{n}$ which is polynomially far from the conjectured threshold of $O(\log n/n)$.

This is different to other threshold problems. For example, absorber-based algorithms were used to bound the threshold for the appearance of the square of a Hamilton cycle in $\mb{G}(n,p)$ up to a subpolynomial factor \cite{KO12}. (Eventually, the true threshold was recovered by Kahn, Narayanan, and Park \cite{KNP21} using ideas related to \cite{FKNP21} and without use of absorbers.) In contrast, we use the richness of possible configurations within $\mc{H}$, which is essentially spread by definition, to show that there is some way to absorb $\mc{L}$ in a spread manner. However, the absorbers are not necessarily themselves contained in $\mc H$. Thus one can think of our absorber as a sparsification template that facilitates boosting the spread, after which the non-algorithmic \cite{FKNP21} is applied. To our knowledge this is the first such use of absorbers.

\subsection{Further directions}

We briefly remark on a few natural questions arising from this work. First, the next two conjectures convey our intuition that the disappearance of uncovered edges probabilistically tells the whole story regarding the appearance of Steiner triple systems. The first conjecture locates a sharp threshold at $2\log n / n$, while the second is the corresponding hitting time statement.

\begin{conjecture}\label{conj:optimistic}
Let $n\in\mb{N}$ satisfy $n\equiv 1,3 \pmod 6$ and fix $\epsilon > 0$. If $\mc{H}\sim\mc{G}^{(3)}(n, (2+\epsilon)\log n/n)$, then with high probability $\mc{H}$ contains an order-$n$ Steiner triple system.
\end{conjecture}

\begin{conjecture}
	Let $n\in\mb{N}$ satisfy $n\equiv 1,3 \pmod 6$ and let $T_1,T_2,\ldots$ be a uniformly random ordering of $\binom{[n]}{3}$. W.h.p.\ the first prefix $T_1,\ldots,T_k$ that covers each $2$-edge at least once contains a Steiner triple system.
\end{conjecture}

We note that a threshold of $O(\log n/n)$ would follow from the existence of an $O(n^{-1})$-spread distribution, while \cref{conj:optimistic} seems to require ideas beyond those in \cite{FKNP21}.

It is also natural to ask for the threshold of more general Steiner systems in random hypergraphs. Since iterative absorption can famously construct designs with arbitrary parameters, we expect that some of the ideas in this paper might extend to this setting. However, to highlight just one potential difficulty, the current argument hinges on the ability to find absorbers that are sufficiently sparse so as not to contribute negatively to the spread. It is unclear whether this is possible for designs with other parameters.

Finally, we wonder whether there exists an efficient algorithm to \textit{find} Steiner triple systems in $\mb{G}^{(3)}(n,p)$, with $p$ sufficiently large for them to exist w.h.p. This question is relevant to Shamir's problem too: while the sharp threshold for the existence of perfect matchings in $\mb{G}^{(3)}(n,p)$ is known, it is not known whether a perfect matching can be found algorithmically. We note that a slight modification of the proof given for \cref{thm:boosting} (though it is not stated this way) allows one to construct the measure on Steiner triple systems which is well-spread in an algorithmic fashion. The difficulty lies in the black-box application of \cite{FKNP21} to show that $\mb{G}^{(3)}(n,n^{-1+o(1)})$ therefore contains Steiner triple systems w.h.p.

\subsection{Organization}

The paper is organized as follows. At the end of this section we introduce some notation. \cref{sec:preliminaries} introduces basic concepts connected to triangle-decompositions and also collects useful probabilistic tools. In \cref{sec:bootstrapping} we lay out the bootstrapping technique that is the heart of our argument. In \cref{sec:IA} we adapt the framework of iterative absorption to the sparse random setting. Finally, as mentioned, in \cref{sec:latin-squares} we describe the modifications to our proof required to obtain \cref{thm:latin}.

\subsection{Notation}\label{sub:notation}

For a graph $G$ we write $V(G)$ for its vertex set, and $G^c$ for its complement within that set. For $v \in V(G)$ we write $N_G(v)$ for its set of neighbors in $G$. If $X\subseteq V(G)$ then $G[X]$ is the induced subgraph on vertex set $X$.

If $\mc{H}$ is a $3$-uniform hypergraph, we may refer to its $3$-edges as triangles. We denote by $e(\mc H)$ the number of triangles in $\mc H$, and we denote by $E(\mc{H})$ the graph of edges that are contained in a triangle of $\mc{H}$.

\section{Preliminaries}\label{sec:preliminaries}

We remind the reader of the following definitions.

\begin{definition}\label{def:triangle-divisible}
A graph $G$ is \textit{triangle-divisible} if every vertex degree is even and the number of edges is a multiple of $3$. A \textit{triangle-decomposition} of a graph $G$ is a collection of triangles in $G$ such that every edge of $G$ is contained in exactly one triangle.
\end{definition}
\begin{remark}
It is easy to see that triangle-divisibility is a necessary but insufficient condition for $G$ to admit a triangle decomposition. The main result of \cite{Kee18} (and also \cite{BGKLMO20}) is that if $G$ is sufficiently large, dense, and typical (pseudorandom in an appropriate sense) then triangle-divisibility is sufficient for $G$ to admit a triangle-decomposition.
\end{remark}

We will repeatedly use the Chernoff bound for binomial and hypergeometric distributions (see for example \cite[Theorems~2.1 and~2.10]{JLR00}) without further comment.
\begin{lemma}[Chernoff bound]\label{lem:chernoff}
Let $X$ be either:
\begin{itemize}
    \item a sum of independent random variables, each of which take values in $\{0,1\}$, or
    \item hypergeometrically distributed (with any parameters).
\end{itemize}
Then for any $\delta>0$ we have
\[\mb{P}[X\le (1-\delta)\mb{E}X]\le\exp(-\delta^2\mb{E}X/2),\qquad\mb{P}[X\ge (1+\delta)\mb{E}X]\le\exp(-\delta^2\mb{E}X/(2+\delta)).\]
\end{lemma}

Next we will require that if a sequence of random variables is stochastically dominated by a sequence of Bernoulli random variables it satisfies an identical set of tail bounds. 
\begin{lemma}[{\cite[Lemma~8]{Ram90}}]\label{lem:binomial-comparison}
Let $X_1,\ldots, X_n$ be $\{0,1\}$-valued random variables such that for all $i\in [n]$, we have that $\mb{P}[X_i = 1|X_1,\ldots, X_{i-1}]\le p$ then $\mb{P}[\sum_{i\in [n]}X_i\ge t]\le \mb{P}[\mr{Bin}(n,p)\ge t]$ for all $t\ge 0$.
\end{lemma}

Finally we will need the symmetric form of the Lov\'asz Local Lemma. 
\begin{lemma}[{\cite[Corollary~5.1.2]{AS16}}]\label{lem:lll}
Let $A_1,A_2,\ldots,A_n$ be events in a probability space such that each $A_i$ is mutually independent of all but $d$ other events. If $\mb{P}[A_i]\le p$ for every $i\in [n]$ and $ep(d+1)\le1$ then $\mb{P}[\bigwedge_{i\in[n]}A_i^c]>0$.
\end{lemma}

\section{Bootstrapping with Spread Families}\label{sec:bootstrapping}

In order to prove \cref{thm:main}, we will iteratively prove the following result for all $\eta\in(0,1]$. We fix a constant $C_{\ref{thm:boosting}} > 0$, large enough for various inequalities we encounter later to hold.

\begin{theorem}[{Theorem($\eta$)}]\label{thm:boosting}
Fix a triangle-divisible graph $G$ on $n$ vertices with $\Delta(G^{c})\le n/\log n$ and $n\ge \exp(C_{\ref{thm:boosting}}/\eta^4)$. Let $\mc{H}\sim \mb{G}^{(3)} (n,n^{\eta}/n)$. With probability at least $1/2$ the collection $\mc{H}$ contains a triangle-decomposition of $G$.
\end{theorem}

Let $\eta' = c(\log n)^{-1/4}$, with $c>0$ a large constant. Note that \cref{thm:boosting}$(\eta')$ implies \cref{thm:main}. Indeed, assuming \cref{thm:boosting}$(\eta')$, if $\mc H_1,\mc H_2,\ldots,\mc H_m$ are (say) $m=\log n$ independent samples of $\mb{G}^{(3)}(n,n^{\eta'} / n)$ then w.h.p.\ at least one of them contains a Steiner triple system. On the other hand $\bigcup_{i=1}^m \mc H_i$ is distributed as $\mb{G}^{(3)}(n,p)$ with $p = \exp(O((\log n)^{3/4})) /n$.

We note that \cref{thm:boosting}$(\eta = 1)$ was proved by Gustavsson \cite{Gus91} and also follows immediately from the results of Keevash \cite[Theorem 2.1]{Kee18} or Barber, K\"uhn, Lo, and Osthus \cite[Theorem 1.2]{BKLO16}. This will serve as the base case for our results, and we will inductively show that the result is true for smaller and smaller $\eta$.

We will need the following result, which can be thought of as stating that a triangle-decomposition iterative absorption scheme that attempts to cover $G\subseteq K_n$ and ultimately has a leftover contained in a smaller set $X \subseteq V(K_n)$ can be performed using only a $|X|^{-1+o(1)}$-fraction of triangles of $G$. (This is essentially the limit for a pure iterative absorption framework as in \cite{GKLO16,BGKLMO20}.) We defer its proof to \cref{sec:IA}. The remaining argument is independent of its justification.

\begin{proposition}\label{prop:IA-random-set}
There exists a constant $C = C_{\ref{prop:IA-random-set}} > 0$ such that the following holds. Let $n \in \mb N$. Fix a subset $X$ of $V(K_n)$ such that $|X| \in [C, n/(\log n)^3]$. Furthermore fix a triangle-divisible graph $G \subseteq K_n$ such that $\Delta(G^c)\le n/\log n + n/(\log n)^2$ and $|X \setminus N_G(v)| \le |X|(1/\log |X|-1/(\log |X|)^2)$ for all $v\in V(G)$. Given a sample $\mc H' \sim \mb{G}^{(3)}(n,(\log |X|)^C/|X|))$, with probability at least $3/4$ there exists an edge-disjoint triangle set $\mc H \subseteq \mc H'$ such that $G^{\ast} \coloneqq E(\mc H)$ satisfies:
\begin{enumerate}
    \item $G\setminus G[X]\subseteq G^{\ast}$ (i.e., $\mc H$ covers all edges of $G$ outside of $X$),
    
    \item $G^{\ast}\subseteq G$ (i.e., $\mc H$ consists of triangles in $G$), and
    
    \item $\Delta((G^c\cup G^{\ast})[X])\le |X|/\log |X|$ (the graph of uncovered edges in $G[X]$ is nearly complete).
\end{enumerate}
\end{proposition}

\begin{figure}[t]
	\includegraphics[width=.8\textwidth]{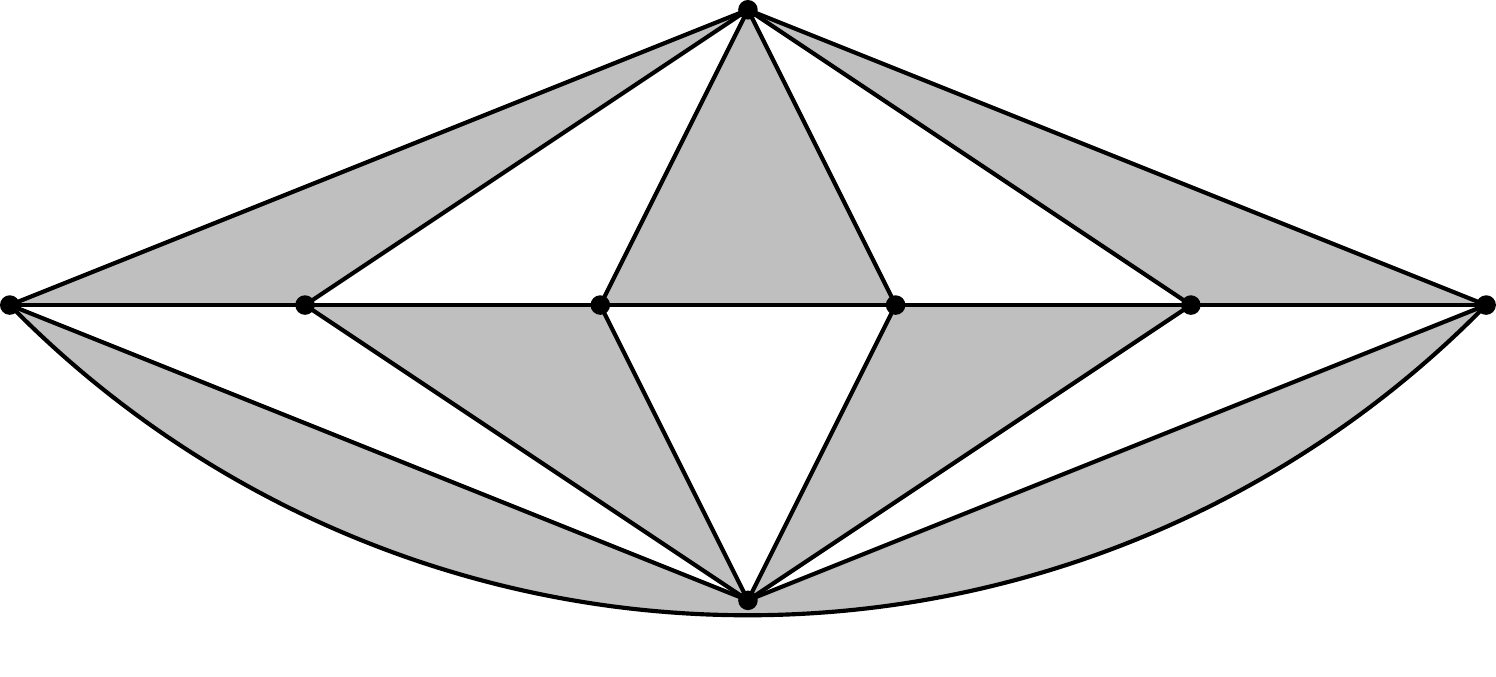}
	\caption{The absorber $\mc{F}_6$ consists of the shaded triangles, while the unshaded triangles (including the external triangle) comprise its absorber flip $\mc{F}_6^\ast$.}\label{fig:F_6}
\end{figure}

As outlined in the introduction, we describe a randomized construction of Steiner triple systems. It uses ``probabilistic absorbers'', whose role is to ``spread out'' the distribution of triangles that are too highly concentrated on a small vertex set. The absorbers are copies of the $3$-uniform hypergraph $\mc{F}_{2m}$, defined as follows: Consider a cycle $C_{2m}$ and add two vertices that are each connected with edges to the $2m$ vertices in the cycle. The resulting graph naturally has a $2$-colorable triangulation. Let $\mc{F}_{2m}$ be one of the color classes. Define its ``absorber-flip'' $\mc{F}_{2m}^\ast$ as the hypergraph on the same vertex set but with the opposite color class of triangles. For an illustration, see \cref{fig:F_6}.

The next lemma allows us to find such absorbers in random linear hypergraphs. Crucially, it requires only that $\mc{F}_{2m}$ be present. It uses the following construction, akin to the R\"odl nibble: Let $G$ be a graph with a distinguished vertex set $X \subseteq V(G)$ and let $p \in [0,1]$. Let $\mb{G}_{\ast}^{(3)}(G,X,p)$ be the distribution on pairs of triangle sets $(\mc{H},\mc{H}')$ defined as follows: First, include each triangle of $G \setminus G[X]$ in $\mc{H}'$ with probability $p$, independently. Then, let $\mc{H} \subseteq \mc{H}'$ be the set of triangles that are edge-disjoint from all other triangles in $\mc{H}'$.

\begin{lemma}\label{lem:hypergraph-comp}
The following holds for a sufficiently large $C_{\ref{lem:hypergraph-comp}}>0$. Fix a graph $G$ on $n \ge C_{\ref{lem:hypergraph-comp}}$ vertices with $\Delta(G^c)\le n/(\log n)$, an integer $2\le m\le (\log n)^{3/4}$, a set $X \subseteq V(G)$ of size at most $n/2$, and a triangle $T$ in $G[X]$. Let $(\mc{H},\mc{H}') \sim \mb{G}_{\ast}^{(3)}(G,X,1/n)$. Then with probability at least $1/(4e)^{12m-6}$, there exists a collection $\mc{C}$ of $2m-1$ triangles in $\mc{H}$ such that $\mc{F} \coloneqq \mc{C}\cup\{T\}$ is a copy of $\mc{F}_{2m}$ with $V(\mc{F})\cap X = V(T)$.
\end{lemma}

\begin{proof}
Let $\mf{F}$ be the collection of copies $\mc{F}$ of $\mc{F}_{2m}$ using triangles from $G$ such that $V(\mc{F})\cap X = V(T)$. Let $Z$ be the number of copies $\mc{F} \in \mf{F}$ such that $\mc{F} \setminus \{T\} \subseteq \mc{H}$. We observe that
\begin{equation}\label{eq:absorber family}
|\mf F| \geq 3(2m-1)! \binom{n/2}{2m-1} \left( 1 - O \left( \frac{m}{\log n} \right) \right).
\end{equation}
This accounts for the three non-isomorphic ways to embed $T$ into $\mc{F}_{2m}$ and the number of ways to choose the remaining vertices. The factor of $1-O(m/\log n)$ accounts for the fact that not all edges of $K_n$ are present in $G$.

Let $\mc{F} \in \mf{F}$ and let $\mc{C} \coloneqq \mc{F} \setminus \{T\}$. We claim that
\begin{equation}\label{eq:absorber prob}
\mb{P} \left[ \mc{C} \subseteq \mc{H} \right] \geq \left( \frac{1}{2e^3n} \right)^{2m-1}.
\end{equation}
Indeed, by definition, $\mb{P} \left[ \mc{C} \subseteq \mc{H}' \right] = n^{-(2m-1)}$. Next, we bound $\mb{P} \left[ \mc{C} \subseteq \mc{H} | \mc{C} \subseteq \mc{H}' \right]$. Conditioning on $\mc{C} \subseteq \mc{H}'$, this is equal to the probability that the $2$-skeleton of $\mc{C}$ does not participate in any triangles of $\mc{H}'$ besides $\mc{C}$. There are at most $|E(\mc{C})| n = (6m-3)n$ such triangles, and so the probability that none of them are in $\mc{H}'$ is at least $\left( 1 - 1/n \right)^{(6m-3)n} > (1.1e)^{-(6m-3)}$. This proves \cref{eq:absorber prob}.

By linearity of expectation, \cref{eq:absorber family}, and \cref{eq:absorber prob} we have that
\begin{equation}\label{eq:absorber expectation}
\mb{E} Z \geq \left( \frac{1}{2e} \right)^{6m-3}.
\end{equation}

Now note that
\[
\mb{E} \left[ Z^2 \right] = \mb{E} Z + \mb{E}[Z(Z-1)] \le \mb{E} Z +  \bigg(3(2m-1)!\binom{n}{2m-1}\bigg)^2n^{-2(2m-1)}\le \mb{E} Z + 9.
\]
In calculating $\mb{E}[Z(Z-1)]$, we have used the fact that any pair of distinct configurations that appear in $\mc{H}$ do not share a triangle. Applying the second moment method we obtain
\[
\mb{P} [Z>0] \geq \frac{(\mb{E} Z)^2}{\mb{E} [Z^2]} \geq \frac{(\mb{E} Z)^2}{\mb{E} Z + 9} \stackrel{\text{\cref{eq:absorber expectation}}}{>} \left( \frac{1}{4e} \right)^{12m-6},
\]
completing the proof.
\end{proof}

The next lemma follows via a direct union-bound computation.

\begin{lemma}\label{lem:hypergraph-comp-2}
Fix a graph $G$ on $n$ vertices and a pair of distinct triangles $T_1,T_2$ in $G$, and sample each triangle in $G$ with probability $1/n$. Call this random hypergraph $\mc{H}$. Let $2\le m\le(\log n)^{3/4}$. Let $q$ be the probability that $\mc{H}\cup\{T_1,T_2\}$ contains two copies, $\mc{F}^1$ and $\mc{F}^2$, of $\mc{F}_{2m}$, with $T_i\in \mc{F}^i$ and $T_{3-i}\notin \mc{F}^i$ for $i\in\{1,2\}$, that share a triangle. There exists a constant $C_{\ref{lem:hypergraph-comp-2}}>0$ such that $q\le C_{\ref{lem:hypergraph-comp-2}} 2^{3m}\!/n^2$ if $T_1$ and $T_2$ are vertex-disjoint and $q\le C_{\ref{lem:hypergraph-comp-2}} 2^{3m}\!/n$ if $T_1$ and $T_2$ share exactly one vertex.
\end{lemma}

\begin{proof}
Consider two copies of $\mc{F}_{2m}$ and all possible ways to identify vertices of one copy to vertices of the other (if there is a repeated triangle as a result of this gluing, we keep only one copy). There are at most $(2^m)^2$ resulting hypergraphs. Suppose that $\mc{F}'$ is obtained in this way and has a repeated triangle. Consider the probability that a copy of $\mc{F}'$ that extends $T_1,T_2$ simultaneously can be found in $\mc{H}$. There are $O(m^2)$ ways to choose which triangles correspond to $T_1,T_2$. Fixing such a choice the probability is bounded by $n^{v(\mc{F}')-6}(1/n)^{e(\mc{F}')-2}$ in the vertex-disjoint case and $n^{v(\mc{F}')-5}(1/n)^{e(\mc{F}')-2}$ when $T_1$ and $T_2$ share a vertex. Therefore, it suffices to show that in both cases $e(\mc{F}')\ge v(\mc{F}')-2$.

We will use the following property of $\mc{F}_{2m}$, which is related to it being an \textit{Erd\H{o}s configuration} in the sense of \cite{GKLO20,KSSS22}. Let $\mc{S} \subseteq \mc{F}_{2m}$ be a set of $s>0$ triangles. Then $\mc{S}$ is incident to at least $s+2$ vertices. Indeed, if $s=2m$ then $\mc{S}=\mc{F}_{2m}$ and so $\mc{S}$ is incident to all $2m+2 = s+2$ vertices. Otherwise, the triangles in $\mc{S}$ cover $s<2m$ edges in the cycle $C_{2m}$. The covered edges form a collection of paths and so have at least $s+1$ vertices. Furthermore, $\mc{S}$ is incident to at least one of the two additional vertices in $\mc{F}_{2m}$, and so $\mc{S}$ is incident to at least $s+2$ vertices, as desired.

Returning to the main argument, let $v$ be the number of pairs of glued vertices and $t$ be the number of triangles that occur as repeated triangles. We have $v(\mc{F}') = 2(2m+2)-v$ and $e(\mc{F}') = 2(2m)-t$, so we equivalently must show $v\ge t+2$. Suppose that the $t$ repeated triangles, within a single copy of $\mc{F}_{2m}$, span $w$ vertices. By the argument above $w \geq t+2$. Moreover, it is evident that $v\ge w$, completing the proof.
\end{proof}

\subsection{Bootstrapping}\label{sub:boosting}

We are now in a position to lay out the bootstrapping argument that establishes iteratively improved versions of \cref{thm:boosting} and, eventually, \cref{thm:main}.

\begin{proof}[Proof of \cref{thm:boosting}]
For $k\ge 0$ let $\eta_k = 2/(k+2)$. As discussed, \cref{thm:boosting}$(\eta=\eta_0)$ follows from \cite[Theorem~1.4]{Kee14} or \cite[Theorem~1.1]{GKLO16}. We will show \cref{thm:boosting}$(\eta = \eta_k)$ via induction on $k$. Then, to see the result for arbitrary $\eta > 0$ we may round down to the nearest $\eta_k$ and appropriately adjust the constant $C_{\ref{thm:boosting}}$.

Let $k\ge 1$ and assume that \cref{thm:boosting}$(\eta=\eta_{k-1})$ holds. Set $\gamma_k = 4/(2k+5)$ and consider a graph $G$ on $n\ge\exp(C_{\ref{thm:boosting}}/\eta_k^4)$ vertices satisfying the conditions of \cref{thm:boosting}$(\eta=\eta_k)$. Our goal is to construct a measure $\mu$ on triangle-decompositions of $G$ which is $O(n^{\eta_k-1}/(\log n)^2)$-spread, since then \cref{thm:FKNP} will imply the result.

Our first task is to construct the absorber. As it concerns the measure $\mu$, the absorber is deterministic. However, we will need it to satisfy certain structural properties. For this reason we will use the probabilistic method. We define the following quantities:
\begin{align*}
&M = n^{1-\gamma_k}\exp \left((\log n)^{1/3} \right),~m = \sqrt{\log n},\\
&\ell_1 = \sqrt{\frac{n}{M^{1+\eta_{k-1}}}},~\ell_2 = \frac{\sqrt{M^{1+\eta_{k-1}}n}}{e^m},~\ell_1' = \frac{\ell_1}{2(\log M)^6}.
\end{align*}

\begin{claim}
	There exist sets $X,X_1,\ldots,X_{\ell_1},Y_1,\ldots,Y_{\ell_1} \subseteq V(G)$ satisfying the following conditions.
	\begin{enumerate}[{\bfseries{Ab\arabic{enumi}}}]
		\item\label{Ab1} $|X|=M$ and $Y_1,\ldots,Y_{\ell_1}$ are disjoint sets of vertices within $V(G)\setminus X$, each of size $\ell_2$, and $X_1,\ldots,X_{\ell_1}\subseteq X$ are sets of vertices of size $|X|/(\log|X|)^2$. We set $G_i = G[X_i\cup Y_i]$.
		\item\label{Ab2} Every triangle in $G[X]$ is contained in at least $\ell_1'$ graphs $G_i$.
		\item\label{Ab3} Every $v\in V(G)$ satisfies $|(V(G)\setminus N_G(v))\cap X|\le|X|(1/\log|X|-2/(\log|X|)^2)$.
		\item\label{Ab4} Every $i\in[\ell_1]$ satisfies $\Delta(G_i^c)\le|V(G_i)|/\log|V(G_i)|$.
	\end{enumerate}
\end{claim}

\begin{proof}
First, let $X \subseteq V(G)$ be a uniformly random set of size $M$. Then, sample $\ell_1$ disjoint sets of vertices $Y_1,\ldots,Y_{\ell_1}$ of size $\ell_2$ uniformly at random from $V(G)\setminus X$. For each $i\in[\ell_1]$ choose a set $X_i \subseteq X$ of size $|X|/(\log |X|)^2$ uniformly at random. Then \cref{Ab1} is satisfied by definition.

Observe that the number of graphs $G_i$ containing a fixed triangle in $G[X]$ is distributed binomially with parameters $(\ell_1,(\log M)^{-2})$. Therefore, by Chernoff's inequality and a union bound, with probability $1-O(|X|^{-1})$ every triangle in $G[X]$ is contained in at least $\ell_1'$ graphs $G_i$. This establishes \cref{Ab2}.

Recall that $\Delta(G^c)\le n/\log n$. Thus, since $X$ was chosen randomly, by a similar application of Chernoff's inequality we conclude that with probability $1-O(|X|^{-1})$ each $v\in V(G)$ satisfies $|(V(G)\setminus N(v))\cap X|\le |X|(1/\log |X|-2/(\log |X|)^2)$. This proves \cref{Ab3}. \cref{Ab4} follows from a similar argument.
\end{proof}

For the remainder of the proof we fix (deterministic) sets $X, X_1,\ldots,X_{\ell_1}, Y_1,\ldots,Y_{\ell_1}$ and graphs $G_1,\ldots,G_{\ell_1}$ satisfying \cref{Ab1,Ab2,Ab3,Ab4}.

We now define $\mu$, which corresponds to the output of a randomized algorithm to find a triangle-decomposition of $G$. At a high level, we (i) sample a rich random \textit{template} of triangles within $X\cup Y_1\cup\cdots\cup Y_{\ell_1}$ to use as an absorber, (ii) use \cref{prop:IA-random-set} to run iterative absorption to cover the remainder apart from $X$ in a spread fashion, (iii) use \cref{thm:boosting}$(\eta=\eta_{k-1})$ to decompose this remainder in a $|X|^{-1+\eta_{k-1}}$-spread fashion, and (iv) simultaneously flip all of the resulting triangles within $X$ using the template to improve the spread. The algorithm is as follows:
\begin{enumerate}[{\bfseries{Alg\arabic{enumi}}}]
    \item\label{Alg1} For each $i\in[\ell_1]$ let $G_i' \coloneqq G_i[X_i\cup Y_i]\setminus G_i[X_i]$ and sample $(\mc{H}_i,\mc{H}_i') \sim \mb{G}_*^{(3)}(G_i',X_i,{|V(G_i')|^{-1}})$ (with independent samples for each $i \in [\ell_1]$).
    
    \item\label{Alg2} Let $\mc{H}_{\mr{IA}} \sim \mb{G}^{(3)}(n,(\log |X|)^{C_{\ref{prop:IA-random-set}}}/|X|)$.
    
    \item\label{Alg3} Sample every triangle in $G[X]$ with probability $|X|^{-1+\eta_{k-1}}$ and call this family $\mc{H}_{\mr{Rand}}$.
    
    \item\label{Alg4} Condition on the event $\mc{E}_{\mr{IA}}$ that there is a triangle set $\mc{H}_{\mr{Dec}} \subseteq \mc{H}_{\mr{IA}}$ satisfying the conditions \cref{prop:IA-random-set}(1-3) for $G' = G\setminus(\bigcup_{i=1}^{\ell_1} E(\mc{H}_i))$.
    
    \item\label{Alg5} Condition on the event $\mc{E}_{\mr{Ind}}$ that $(G'\setminus E(\mc{H}_{\mr{Dec}}))[X]$ has a triangle-decomposition $\mc{H}_{\mr{Ind}} \subseteq \mc{H}_{\mr{Rand}}$.
    
    \item\label{Alg6} Condition on the event $\mc{E}_{\mr{Abs}}$ that we can find, for each $T\in\mc{H}_{\mr{Ind}}$, an index $i_T\in[\ell_1]$ and a $3$-uniform hypergraph $F_T\simeq\mc{F}_{2m}$ contained in $\mc{H}_{i_T} \cup \{T\}$ such that (a) $V(F_T)\cap X = V(T)$ and (b) the triangle sets $\{F_T\}_{T\in\mc{H}_{\mr{Ind}}}$ are edge-disjoint.
    
    \item\label{Alg7} Let $\mc{H}_{\mr{Abs}} = \bigcup_{T\in\mc{H}_{\mr{Ind}}} (F_T \setminus \{T\})$ and let $\mc{H}_{\mr{Flip}} = \bigcup_{T\in\mc{H}_{\mr{Ind}}} F_T^\ast$, where for each $T\in\mc{H}_{\mr{Ind}}$ the $3$-graph $F_T^\ast$ consists of all triangles of the corresponding ``absorber-flip'' $\mc{F}_{2m}^\ast$ associated to $F_T$.
    
    \item\label{Alg8} Finally, output the triple system $\mc{H} = \mc{H}_{\mr{Dec}}\cup\mc{H}_{\mr{Flip}}\cup(\bigcup_{i=1}^{\ell_1}\mc{H}_i\setminus\mc{H}_{\mr{Abs}})$.
\end{enumerate}

\begin{remark}
Although we have described an algorithm to \textit{sample} from $\mu$, this should not be misconstrued as an algorithm to \textit{find} the triangle-decomposition guaranteed by \cref{thm:boosting}. The main reason is that we will ultimately use the non-algorithmic \cref{thm:FKNP} applied to $\mu$.

Nevertheless, with a slight modification to the algorithm one can efficiently sample from $\mu$. To do so, in \cref{Alg5}, one should inductively invoke the ability to efficiently sample from an $|X|^{-1+\eta_{k-1}}$-spread distribution on triangle-decompositions. However, we would still rely on \cref{thm:FKNP} to convert the spread distribution into a threshold result. Additionally, the analysis is slightly more involved. As thresholds, rather than spread distributions, are our main concern, we have not made this change.
\end{remark}

To verify that $\mu$ satisfies the desired properties, we must check that $\mu$ is supported on triangle-decompositions and also that the above process succeeds with nonzero probability (otherwise $\mu$ is not well-defined). We do so in \cref{clm:decomp,clm:success}. Finally, in \cref{clm:spread} we show that $\mu$ has the appropriate spread.

\begin{claim}\label{clm:decomp}
$\mc{H}$ is a triangle-decomposition of $G$.
\end{claim}

\begin{proof}
First note that $\bigcup_{i=1}^{\ell_1}\mc{H}_i$ is a set of edge-disjoint triangles contained in $G$. Indeed, every $\mc{H}_i$ is such a set by definition. Additionally, the sets $Y_1,\ldots,Y_{\ell_1}$ are disjoint, and by definition every edge in $G_i'$ contains at least one vertex from $Y_i$. Hence, the triangle sets $\mc{H}_1,\ldots,\mc{H}_{\ell_1}$ are mutually edge-disjoint.

Next, by \cref{Alg4} we have that $\mc{H}_{\mr{Dec}}$ is a set of edge-disjoint triangles covering all edges of $G'$ except for some in $X$. By \cref{Alg5} these remaining edges are covered by $\mc{H}_{\mr{Ind}}$. That is, $\mc{H}_{\mr{Dec}} \cup\mc{H}_{\mr{Ind}} \cup\bigcup_{i=1}^{\ell_1}\mc{H}_i$ is a triangle-decomposition of $G$. Finally, we have $E(F_T) = E(F_T^\ast)$ and the $F_T$ are edge-disjoint by \cref{Alg6}. Hence $\mc{H}_{\mr{Ind}} \cup\mc{H}_{\mr{Abs}}$ and $\mc{H}_{\mr{Flip}}$ are each edge-disjoint triangle-decompositions, and they decompose the same underlying graph. Since $\mc{H}_{\mr{Abs}}\subseteq\bigcup_{i=1}^{\ell_1}\mc{H}_i$ by \cref{Alg6}, the result follows.
\end{proof}

Next we show that $\mu$ is well-defined. In fact, we show that the above process defining $\mu$ succeeds with probability at least $(3/4)(1/2)(2/3) = 1/4$.

\begin{claim}\label{clm:success}
We have $\mb{P}[\mc{E}_{\mr{IA}}]\ge 3/4$, $\mb{P}[\mc{E}_{\mr{Ind}}|\mc{E}_{\mr{IA}}]\ge 1/2$, and $\mb{P}[\mc{E}_{\mr{Abs}}|\mc{E}_{\mr{IA}}\wedge\mc{E}_{\mr{Ind}}]\ge 2/3$.
\end{claim}

\begin{proof}
We first claim that $G'$ and $X$ satisfy the assumptions of \cref{prop:IA-random-set} regardless of the outcome of $\bigcup_{i=1}^{\ell_1}\mc{H}_i$. Note that by \cref{Alg1,Alg4}, $G'$ differs from $G$ only by edges in $\bigcup_{i=1}^{\ell_1} E(G_i')$. Additionally, since $Y_1,\ldots,Y_{\ell_1}$ are disjoint, every $v\in V(G)\setminus X$ is contained in at most one graph $G_i'$. By construction, every $v \in Y_i$ has at most $|X_i| = |X|/(\log|X|)^2$ neighbors in $X$ within $G_i'$. Therefore $|X\setminus N_{G'}(v)|\le|X|(1/\log|X|-2/(\log|X|)^2) + |X|/(\log|X|)^2$, using \cref{Ab3}. For $v\in X$, the number of neighbors within $X$ is unchanged upon removing $\bigcup_{i=1}^{\ell_1} E(G_i')$. Additionally, $C_{\ref{prop:IA-random-set}}\le|X|\le n/(\log n)^3$ is immediate from \cref{Ab1} and the assumption $n \geq \exp(C_{\ref{thm:boosting}}/\eta_k^4)$. Finally,
\[
\Delta((G')^c) \le \Delta(G^c)+|X|+ \sum_{i=1}^{\ell_1} |Y_i| \stackrel{\text{\cref{Ab1}}}{\leq} \Delta(G^c) + n/(\log n)^2 \leq n / \log n + n / (\log n)^2.
\]
Thus all conditions are satisfied, and by \cref{prop:IA-random-set} the necessary $\mc{H}_{\mr{Dec}}\subseteq\mc{H}_{\mr{IA}}$ exists with probability at least $3/4$.

We condition on $\mc{E}_{\mr{IA}}$ occurring. This means that $(G'\setminus E(\mc{H}_{\mr{Dec}}))[X]$ has maximum degree at most $|X|/\log|X|$ due to \cref{prop:IA-random-set}(3). Therefore the inductive hypothesis \cref{thm:boosting}$(\eta=\eta_{k-1})$ applies and shows that the necessary $\mc{H}_{\mr{Ind}}\subseteq\mc{H}_{\mr{Rand}}$ exists with probability at least $1/2$, since
\[|X|\ge n^{1-\gamma_k}\ge(\exp(C_{\ref{thm:boosting}}/\eta_k^4))^{1-\gamma_k}\ge\exp(C_{\ref{thm:boosting}}/\eta_{k-1}^4).\]

It remains to prove that $\mb{P}[\mc{E}_{\mr{Abs}}|\mc{E}_{\mr{IA}} \wedge \mc{E}_{\mr{Ind}}] \ge 2/3$. The proof is more involved, and we break it into three steps.

\textbf{Step 1: Potential absorbers with few overlaps.}
We first establish that certain conditions hold with high probability in the unconditional model. For any triangle $T$ of $G[X]$ and any $i\in[\ell_1]$, let $\mf{A}_{T,i}$ be the set of copies $\mc F$ of $\mc F_{2m}$ that contain $T$, use only triangles in $\mc H_i \cup \{T\}$, and $V(\mc F) \cap X = V(T)$. Note that this is a random collection depending only on $\mc{H}_i$. Given $T$ and $i$, we consider (a) the number of choices $N_{T,i} = |\mf{A}_{T,i}|$ and (b) the number $M_{T,i}$ of triangles $T'$ of $G[X]$ such that $E(T')\cap E(T) = \emptyset$ and $\mf{A}_{T,i},\mf{A}_{T',i}$ contain a pair of copies of $\mc{F}_{2m}$ that are not edge-disjoint (equivalently, these copies share a triangle). Let $\mc{E}_{\mr{Tem}}$ be the event that the following conditions for the random template hold:
\begin{enumerate}[{\bfseries{Tem\arabic{enumi}}}]
    \item\label{Tem1} For each triangle $T$ of $G[X]$, $|\{i\in[\ell_1]: N_{T,i} > 0\}| \ge\ell_1'/(8e)^{12m-6}$.
    \item\label{Tem2} For each $T$, at most $\ell_1'/(32e)^{12m-6}$ indices $i\in[\ell_1]$ satisfy $M_{T,i}\ge 2^{90m}|X|^2\!/\ell_2$.
\end{enumerate}
We claim that $\mb{P}[\mc{E}_{\mr{Tem}}]\ge 1-1/n$. For \cref{Tem1}, fix $T$ and note that $N_{T,i}$ for $i\in[\ell_1]$ are independent, as each depends only on $\mc{H}_i$. Furthermore, \cref{lem:hypergraph-comp} implies that $\mb{P}[N_{T,i}>0]\ge 1/(4e)^{12m-6}$ whenever $V(T)\subseteq X_i$. By \cref{Ab2} we have at least $\ell_1'$ such indices, so Chernoff's inequality implies that $|\{i\in[\ell_1]: N_{T,i} > 0\}| \ge \ell_1'/(8e)^{12m-6}$ with probability at least $1-n^{-5}$. Applying a union bound, \cref{Tem1} holds with probability at least $1-n^{-2}$.

For \cref{Tem2}, fix $T$ and note that by linearity of expectation and \cref{lem:hypergraph-comp-2} we have
\[\mb{E}M_{T,i}\le|X|^2\cdot O(2^{3m}\!/\ell_2) + |X|^3\cdot O(2^{3m}\!/\ell_2^2)\le 2^{4m}|X|^2\!/\ell_2,\]
since every triangle in $G[X]$ shares a vertex with at most $|X|^2$ other triangles in $G[X]$. By Markov's inequality, we have
\[\mb{P}[M_{T,i}\ge 2^{90m}|X|^2\!/\ell_2]\le 2^{-86m}.\]
Now let $Z_{T}$ be the number of indices $i \in [\ell_1]$ satisfying $M_{T,i} \geq 2^{90m}|X|^2\!/\ell_2$. Applying Chernoff's inequality:
\[
\mb P [Z_{T} \geq \ell_1'/(32e)^{12m-6}] \leq \frac{1}{n^5}.
\]
By a union bound, \cref{Tem2} holds with probability at least $1-n^{-2}$. Thus $\mc{E}_{\mr{Tem}}$ holds with probability at least $1-1/n$, as desired.

\textbf{Step 2: Few overlaps in absorbers for $\mc{H}_{\mr{Rand}}$.}
Next we show that the number of potential conflicts counted by $M_{T,i}$ significantly diminishes (w.h.p.) when we consider only the sparse random set $\mc{H}_{\mr{Rand}}$. First, for each triangle $T$ in $G[X]$ we define the index set $I_T$ of indices $i$ satisfying $N_{T,i} > 0$ and $M_{T,i} < 2^{90m}|X|^2/\ell_2$. If $\mc{E}_{\mr{Tem}}$ holds then $|I_T|\ge\ell_1'/(32e)^{12m-6}$ for all $T$.

Next, given $T$, let $\mc{M}_T'$ be the set of triangles $T'\in\mc{H}_{\mr{Rand}}$ with $E(T')\cap E(T) = \emptyset$ such that for some $i\in I_T$, the collections $\mf{A}_{T,i},\mf{A}_{T',i}$ contain a pair of copies of $\mc{F}_{2m}$ that are not edge-disjoint. Let $\mc{E}_{\mr{Pack}}$ be the event that the following holds:
\begin{enumerate}[{\bfseries{Pack\arabic{enumi}}}]
    \item\label{Pack1} For all $T\in\mc{H}_{\mr{Rand}}$, we have $|\mc{M}_T'|\le 2^{95m}|X|^{1+\eta_{k-1}}\ell_1/\ell_2$.
\end{enumerate}
We claim that $\mb{P}[\mc{E}_{\mr{Pack}}|\mc{E}_{\mr{Tem}}]\ge 1-1/n$. Indeed, reveal all of $\bigcup_{i=1}^{\ell_1}\mc{H}_i$ and fix any triangle $T$ of $G[X]$. Condition on $T \in \mc{H}_{\mr{Rand}}$ and on $\mc{E}_{\mr{Tem}}$. Then, expose the remainder of $\mc{H}_{\mr{Rand}}$. It follows that $|\mc{M}_T'|$ is distributed as $\mr{Bin}(M^\ast,|X|^{-1+\eta_{k-1}})$ for some
\[M^\ast\le\sum_{i\in I_T}M_{T,i}\le 2^{90m}|X|^2\ell_1/\ell_2.\]
Chernoff's inequality and a union bound now imply that $\mb{P} [\mc{E}_{\mr{Pack}} | \mc{E}_{\mr{Tem}}]\ge 1-1/n$, as desired.

\textbf{Step 3: Finding simultaneous edge-disjoint absorbers.}
Now we condition on $\mc{E}_{\mr{IA}}$ and $\mc{E}_{\mr{Ind}}$ occurring. We have
\[
\mb{P}[(\mc{E}_{\mr{Tem}}\wedge\mc{E}_{\mr{Pack}})^c|\mc{E}_{\mr{IA}}\wedge\mc{E}_{\mr{Ind}}]\le \frac{O(1/n)}{1/2\cdot 3/4} = O(1/n),
\]
so $\mb{P}[\mc{E}_{\mr{Tem}} \wedge \mc{E}_{\mr{Pack}} | \mc{E}_{\mr{IA}}\wedge\mc{E}_{\mr{Ind}}]\ge 1 - O(1/n)$.

We now argue that if $\mc{E}_{\mr{Tem}}\wedge\mc{E}_{\mr{Pack}}\wedge\mc{E}_{\mr{IA}} \wedge \mc{E}_{\mr{Ind}}$ holds then $\mc{E}_{\mr{Abs}}$ holds, which will finish the proof. Assuming $\mc{E}_{\mr{Ind}}$, \cref{Alg5} succeeds so there exists some edge-disjoint collection $\mc{H}_{\mr{Ind}}\subseteq\mc{H}_{\mr{Rand}}$. We will use the Lov\'asz Local Lemma (\cref{lem:lll}) to prove the existence of the absorbers necessary for \cref{Alg6}. As we are assuming \cref{Tem1,Tem2}, for each $T \in \mc{H}_{\mr{Ind}}$ we have a nonempty set of indices $I_T$ such that for all $i\in I_T$, $N_{T,i} > 0$. For every $T \in \mc{H}_{\mr{Ind}}$ we choose, uniformly at random, an index $i_T \in I_T$ and then uniformly at random choose one of the extensions of $T$ (isomorphic to $\mc{F}_{2m}$) counted by $N_{T,i_T}$. We make these choices independently for each triangle in $\mc{H}_{\mr{Ind}}$. We claim that with nonzero probability, all of these extensions are edge-disjoint.

We define a ``disjointness graph'' $H$ with vertex set $\mc{H}_{\mr{Ind}}$. For each $T,T'\in\mc{H}_{\mr{Ind}}$, we put an edge between them if there is some $i\in I_T\cap I_{T'}$ such that $\mf{A}_{T,i},\mf{A}_{T',i}$ contain a pair of copies of $\mc{F}_{2m}$ that are not edge-disjoint. We see that $N_H(T)\subseteq\mc{H}_{\mr{Ind}}\cap\mc{M}_T'$. Hence, by \cref{Pack1},
\[
\Delta(H)\le 2^{95m}|X|^{1+\eta_{k-1}}\ell_1/\ell_2 = 2^{90m}e^m = e^{O(\sqrt{\log n})}.
\]

For each edge $f\in E(H)$, let $\mc{B}_f$ be the ``bad'' event that the random extensions chosen for $T$ and $T'$ share an edge. We wish to show that with nonzero probability, we can simultaneously avoid all the bad events. This will prove the result, since by definition the only pairs $T,T'$ that can have a conflict with this process are those corresponding to some $f\in E(H)$. To apply the Lov\'asz Local Lemma we observe that each $\mc{B}_f$ is mutually independent from all other events except for $\mc{B}_{f'}$ where $f,f'$ share a vertex. There are at most $2\Delta(H)$ such events. Additionally, for each bad event $\mc{B}_f$ where $f = \{T,T'\}$ we have
\[
\mb{P}[\mc{B}_f]\le \mb{P} [i_T = i_{T'}] \leq \frac{1}{|I_T|} \le \frac{(32e)^{12m-6}}{\ell_1'}.
\]
Since $\ell_1\ge n^{1/(20k^2)}$, we see that $\ell_1' = n^{\Omega(1)}$ and therefore
\[
e\max_{f\in E(H)}\mb{P}[\mc{B}_f]\cdot(2\Delta(H)+1) = \frac{e^{O(\sqrt{n})}}{n^{\Omega(1)}} < 1.
\]
The result follows.
\end{proof}

Finally, we verify the spread condition.
\begin{claim}\label{clm:spread}
$\mu$ is $O(n^{\eta_k-1}/(\log n)^2)$-spread.
\end{claim}
\begin{proof}
\textbf{Step 1: Spread of $\mc{H} \setminus \mc{H}_{\mr{Flip}}$.}
Let $\mc T = \{T_1,\ldots,T_t\}$ be a set of triangles in $G$, and consider the probability that they simultaneously appear in a sample $\mc{S} \sim \mu$. Since $\mc{S}$ consists of edge-disjoint triangles we may assume that the triangles in $\mc T$ are edge-disjoint as well. Additionally, recall that the underlying probability space associated to $\mu$ is defined by independent samples $\mc{H}_{\mr{IA}}$, $\mc{H}_{\mr{Rand}}$, and $\mc{H}_1',\ldots,\mc{H}_{\ell_1}'$ conditional on $\mc{E}_{\mr{IA}}\wedge\mc{E}_{\mr{Ind}}\wedge\mc{E}_{\mr{Abs}}$. Finally, we have $\mc{H} = \mc{H}_{\mr{Dec}}\cup \mc{H}_{\mr{Flip}}\cup(\bigcup_{i=1}^{\ell_1}\mc{H}_i\setminus\mc{H}_{\mr{Abs}})$.

Suppose that $\mc T \subseteq \mc S$. Then each $T_j$ is either in $\mc{H}_{\mr{Dec}}\subseteq\mc{H}_{\mr{IA}}$, in $\bigcup_{i=1}^{\ell_1}\mc{H}_i\setminus\mc{H}_{\mr{Abs}}\subseteq\bigcup_{i=1}^{\ell_1}\mc{H}_i'$, or in $\mc{H}_{\mr{Flip}}$. Let $J_1,J_2,J_3 \subseteq [t]$ be the index sets of the triangles contained in each of these respective sets. Observe that for triangles $T_j$ such that $j\in J_3$, this means that $T_j$ is in some ``absorber-flip'' $\mc{F}_{2m}^\ast$ where the corresponding $\mc{F}_{2m}$ consists of a triangle $S_j\in\mc{H}_{\mr{Ind}}\subseteq\mc{H}_{\mr{Rand}}$ and $2m-1$ triangles of $\bigcup_{i=1}^{\ell_1}\mc{H}_i'$ by \cref{Alg6,Alg7}. Furthermore, those triangles cannot share an edge with any $\{T_j: j\in J_2\}$.

Now, for a partition $J_1\sqcup J_2\sqcup J_3=[t]$, let $\mc{E}_{J_1,J_2,J_3}$ be the event that:
\begin{enumerate}[{\bfseries{E\arabic{enumi}}}]
    \item\label{E1} $T_j\in\mc{H}_{\mr{IA}}$ for all $j\in J_1$;
    \item\label{E2} $T_j\in\bigcup_{i=1}^{\ell_1}\mc{H}_i'$ for all $j\in J_2$;
    \item\label{E3} For all $j\in J_3$, there is a copy of $F_j \simeq \mc{F}_{2m}$, consisting of one triangle $S_j \in \mc{H}_{\mr{Rand}}$ and $2m-1$ triangles of $\bigcup_{i=1}^{\ell_1}\mc{H}_i'$ such that $T_j$ is in the associated ``absorber-flip'' $\mc{F}_{2m}^\ast$;
    \item\label{E4} $F_j$ is edge-disjoint from $\{T_{j'}: j'\in J_2\}$;
    \item\label{E5} Every pair $F_j,F_{j'}$ for distinct $j,j'\in J_3$ is edge-disjoint or identical.
\end{enumerate}
The above analysis shows that the union of the $3^t$ events $\mc{E}_{J_1,J_2,J_3}$ covers all possible situations where $\{T_1,\ldots,T_t\}\subseteq\mc{S}$. Thus, we have
\[\mb{P}[\{T_1,\ldots,T_t\}\subseteq\mc{S}]\le\sum_{J_1,J_2,J_3}\mb{P}[\mc{E}_{J_1,J_2,J_3}|\mc{E}_{\mr{IA}}\wedge\mc{E}_{\mr{Ind}}\wedge\mc{E}_{\mr{Abs}}]\le 4\cdot 3^t\max_{J_1,J_2,J_3}\mb{P}[\mc{E}_{J_1,J_2,J_3}]\]
by \cref{clm:success} and Bayes' theorem. Furthermore, the event $\mc{E}_{J_1,J_2,J_3}$ does not depend directly on $\mc S$, but rather on an independent model of triangles. Thus, we can essentially disregard the complicated process \cref{Alg1,Alg2,Alg3,Alg4,Alg5,Alg6,Alg7,Alg8} in favor of this substantially simpler situation.

We will reduce the situation further to studying the triangles in $J_3$. Given $J_3\subseteq[t]$, let $\mc{E}_{J_3}$ be the event that \cref{E3,E5} hold. We see that
\[
\mb{P}[\mc{E}_{J_1,J_2,J_3}]\le\bigg(\prod_{j\in J_1}\mb{P}[T_j\in\mc{H}_{\mr{IA}}]\prod_{j\in J_2}\mb{P}\bigg[T_j\in\bigcup_{i=1}^{\ell_1}\mc{H}_i'\bigg]\bigg)\mb{P}[\mc{E}_{J_3}].
\]
Indeed, the first term can be extracted due to independence of $\mc{H}_{\mr{IA}}$ and conditions \cref{E2,E3,E4,E5}. Additionally, careful scrutiny shows that in fact \cref{E2} is independent from the event that events \cref{E3,E4,E5} hold by construction.

By definition, we have $\mb{P}[T_j\in\mc{H}_{\mr{IA}}] = (\log|X|)^{C_{\ref{prop:IA-random-set}}}/|X|$ and $\mb{P}[T_j\in\bigcup_{i=1}^{\ell_1}\mc{H}_i']\le 1/|V(G_i')| \le 1/|X|$. Since $|X| = n^{1-\gamma_k}\exp((\log n)^{1/3})$, these terms are bounded by $n^{\eta_k-1}/(\log n)^2$. Putting everything together, we find
\begin{equation}\label{eq:flip-spread-reduction}
\mb{P}[\{T_1,\ldots,T_t\}\subseteq\mc{S}]\le 4\cdot 3^t\max_{J\subseteq[t]}(n^{\eta_k-1}/(\log n)^2)^{t-|J|}\mb{P}[\mc{E}_J].
\end{equation}

\textbf{Step 2: Preliminaries for understanding $\mc{H}_{\mr{Flip}}$.}
Now we analyze the absorber-flips. Call a copy of $\mc{F}_{2m}$ with a distinguished triangle a \emph{rooted absorber}, and call any nonempty subset of triangles $\mc{P}\subseteq\mc{F}_{2m}^\ast$ within the flip of a rooted absorber a \emph{polymer}. If we additionally distinguish a triangle of a polymer we call it a \emph{rooted polymer}. We see that there are $2^{2m}-1$ (labeled) polymers. At a high level, we wish to count the number of ways to break $\{T_j: j\in J\}$ into polymers, then count extensions to a full $\mc{F}_{2m}^\ast$, and then consider the probability that the corresponding $\mc{F}_{2m}$ is in $\mc{H}_{\mr{Rand}}\cup\bigcup_{i=1}^{\ell_1}\mc{H}_i'$.

Let
\[p(s) = \max_{T_1,\ldots,T_t}\max_{J\in\binom{[t]}{s}}\mb{P}[\mc{E}_J],\]
where the maximum is over edge-disjoint collections of triangles.

Given a triangle $T$ of $G_i'$ for some $i$ and an edge-disjoint triangle set $\mc{T}$, both within $K_n$, and given a rooted polymer $\mc{P}$, let $f(\mc{P},T,\mc{T})$ be the number of ways to extend $T$ to a copy of $\mc{P}$ within $G_i'$ where $T$ is the root of the polymer, and the other triangles are all in $\mc{T}$. We claim that
\begin{equation}\label{eq:fPTT bound}
f(\mc P, T, \mc T) \leq
\begin{cases}
(2\ell_2)^{v(\mc{P})-e(\mc{P})-2} & e(\mc{P})\le 2m-1\\
(2\ell_2) & e(\mc{P}) = 2m.
\end{cases}
\end{equation}
(Here, $v(\mc{P})$ is the number of vertices incident to triangles in $\mc P$.)

We first reduce to the case where $e(\mc P) \leq 2m-1$. Indeed, if $e(\mc P) = 2m$ let $\mc P'$ be a polymer obtained by removing a non-root triangle from $\mc P$. It then holds that $f(\mc P,T,\mc T) \leq f(\mc P',T,\mc T)$ and $v(\mc{P}')-e(\mc{P}')-2 = 2m+2-(2m-1)-2 = 1$. Thus, \cref{eq:fPTT bound} with $e(\mc P) = 2m$ follows from \cref{eq:fPTT bound} with $e(\mc P) = 2m-1$.

We prove \cref{eq:fPTT bound} when $e(\mc{P})\le 2m-1$ by induction on $e(\mc{P})$. When $e(\mc{P})=1$, the corresponding polymer consists only of the root triangle and hence must be $\{T\}$; the result follows. Now assume that \cref{eq:fPTT bound} holds for all $e(\mc{P}) < k$ where $2\le k\le 2m-1$. Let $\mc{P}$ be a polymer with $e(\mc P) = k$. Now fix a rooted polymer $\mc{P}' \subseteq \mc P$ such that $e(\mc{P}\setminus\mc{P}') = 1$ and $v(\mc{P})-v(\mc{P}')\ge 1$ (this is possible since $2\le e(\mc{P})\le 2m-1$: any such polymer covers a proper subgraph of the cycle $C_{2m}$, and we can remove a non-root triangle containing a degree $1$ vertex from $\mc P$). By the inductive hypothesis
\[f(\mc{P}',T,\mc{T})\le(2\ell_2)^{v(\mc{P}')-e(\mc{P}')-2}.\]
We now consider the possible ways to extend a copy of $\mc P'$ to a copy of $\mc P$, using only triangles from $\mc T$. There are three cases. If $v(\mc{P})-v(\mc{P}')=3$, note that since the triangles in $\mc T$ are edge-disjoint there are at most $|\mc T| \leq |V(G_i')|^2 \leq (2\ell_2)^2$ possible extensions. If $v(\mc{P})-v(\mc{P}')=2$, there are at most $(2\ell_2)$ possible triangles in $\mc{T}$ within $G_i'$ which could be in $\mc{P}\setminus\mc{P}'$ as this triangle must contain a fixed vertex (given $\mc{P}'$). Finally, if $v(\mc{P})-v(\mc{P}')=1$, note that there is at most $1$ triangle in $\mc{T}$ within $G_i'$ which could be in $\mc{P}\setminus\mc{P}'$ as this triangle must contain a fixed edge (given $\mc{P}'$) and the triangles in $\mc{T}$ are edge disjoint. This completes the inductive proof.

Next, given a polymer $\mc{P}$ within some $G_i'$, let $g(\mc{P})$ be the number of ways to extend it to a full copy of $\mc{F}_{2m}^\ast$ within $G_i'$ with its flip having the property that at least one triangle is fully within $X$. We claim that
\[g(\mc{P})\le(2\ell_2)^{2m-v(\mc{P})+2}\]
and that if $e(\mc{P}) = 1$ then $g(\mc{P})\le|X|(2\ell_2)^{2m-2}$. To see this note that given the polymer there are at most $2m-v(\mc{P})+2$ labelled vertices to be specified and at most $(2\ell_2)$ choices for each vertex. Additionally, when $e(\mc{P}) = 1$, at least one vertex remaining to be chosen must be in $X$ (since the flip of $\mc{F}_{2m}^\ast$ has at least one triangle fully in $X$, while $\mc F_{wm}^\ast$ does not), improving the bound to $|X|\cdot (2\ell_2)^{2m+2-3-1} = |X|(2\ell_2)^{2m-2}$ as desired.

\textbf{Step 3: Spread of $\mc{H}_{\mr{Flip}}$.}
Finally, we bound $p(s)$. Suppose that $J \in \binom{[t]}{s}$. For each $j\in J$ in increasing order, there are at most $2^{2m}-1$ ways to choose which polymer type $\mc{P}$ the set $\{T_{j'}: j'\in J,~F_j = F_{j'}\}$ creates. Given $\mc P$, there are at most $f(\mc{P},T_j,\{T_{j'}: j'\in J\})$ ways to choose how $T_j$ actually extends to that polymer within $\{T_{j'}: j'\in J\}$. There are then at most $g(\mc{P})$ ways to count the number of extensions to a full $\mc{F}_{2m}^\ast$. Given these choices, the probability that the flip of this $\mc F_{2m}^\ast$ is contained in $\mc{H}_{\mr{Rand}} \cup \bigcup_{i=1}^{\ell_1}\mc{H}_i'$ is at most $|X|^{\eta_{k-1}-1}(1/\ell_2)^{2m-1}$.

It follows that
\begin{align*}
p(s)&\le\sum_{s'=1}^{2m}\binom{2m}{s'}\Big(\max_{\substack{e(\mc{P})=s'\\T,\mc{T}}}f(\mc{P},T,\mc{T})g(\mc{P})\cdot|X|^{\eta_{k-1}-1}(1/\ell_2)^{2m-1}\Big)p(s-s'),\\
&\le\max_{\substack{s'\in[2m]\\e(\mc{P})=s'\\T,\mc{T}}}(2m)^{s'}f(\mc{P},T,\mc{T})g(\mc{P})\cdot|X|^{\eta_{k-1}-1}(1/\ell_2)^{2m-1}p(s-s')
\end{align*}
where we let $p(s) = 0$ for $s < 0$ and $p(0) = 1$.

When $e(\mc{P}) = s'\in\{2,\ldots,2m-1\}$ we have
\[f(\mc{P},T,\mc{T})g(\mc{P})\cdot|X|^{\eta_{k-1}-1}(1/\ell_2)^{2m-1}\le(2\ell_2)^{1-s'}\cdot|X|^{\eta_{k-1}-1}\le(n^{\eta_k-1}/(\log n)^3)^{s'}.\]
The last inequality is true since it holds for $s' = 2$ and since $2\ell_2 \geq n^{1-\eta_k}(\log n)^3$.

When $e(\mc{P}) = 2m$ we have
\[f(\mc{P},T,\mc{T})g(\mc{P})\cdot|X|^{\eta_{k-1}-1}(1/\ell_2)^{2m-1}\le 2\ell_2^{2-2m}\cdot|X|^{\eta_{k-1}-1}\le(n^{\eta_k-1}/(\log n)^3)^{2m}.\]
The inequality holds as $|X|^{\eta_{k-1}-1}\le 1$ and $2\ell_2^{(2-2m)/(2m)}\le\exp(\sqrt{\log n})/(2\ell_2)\le n^{\eta_{k}-1}/(\log n)^3$ as $n^{\Omega(1/k^2)}\ge\exp(O(m))$.

When $e(\mc{P}) = 1$ we have 
\begin{align*}
f(\mc{P},T,\mc{T})g(\mc{P})&\cdot|X|^{\eta_{k-1}-1}(1/\ell_2)^{2m-1} \le |X|^{\eta_{k-1}}2^{2m-2}/\ell_2 \\
&\le (4e)^{m}/(n^{1/2}|X|^{(1-\eta_{k-1})/2})\le n^{\eta_k-1}/(\log n)^3
\end{align*}
where the final inequality follows as $n^{\eta_k-1/2}|X|^{(1-\eta_{k-1})/2} = n^{\Omega(1/k^2)}\ge\exp(O(m))$.

Putting this together, we obtain
\[p(s)\le\max_{s'\in[2m]}(4mn^{\eta_k-1}/(\log n)^3)^{s'}p(s-s').\]
Along with the initial conditions, this immediately yields $p(s)\le(4mn^{\eta_k-1}/(\log n)^3)^s$. Finally, combining with \cref{eq:flip-spread-reduction} yields
\[\mb{P}[\{T_1,\ldots,T_t\}\subseteq\mc{S}]\le(O(n^{\eta_k-1}/(\log n)^2))^t,\]
as desired.
\end{proof}

\cref{clm:decomp,clm:spread} imply that $\mu$ is an $O(n^{\eta_k-1/(\log n)^2})$-spread distribution on triangle-decompositions of $G$. Applying \cref{thm:FKNP} to $\mu$ yields \cref{thm:boosting}$(\eta=\eta_k)$, completing the induction.
\end{proof}

\section{Iterative Absorption in Random Hypergraphs}\label{sec:IA}

In this section we use the machinery of iterative absorption to prove \cref{prop:IA-random-set}. Informally, it states that given a nearly complete graph $G$ and a specified set $X \subseteq V(G)$, one can use edge-disjoint triangles to cover all edges in $G \setminus G[X]$ while only covering a small fraction of edges in $G[X]$. Moreover, the triangles can be restricted to a sparse random set. It is proved via iterating the following lemma.

\begin{lemma}\label{lem:IA-iteration-simplified}
There exists a constant $C_{\ref{lem:IA-iteration-simplified}} >0$ such that the following holds. Let $n \in \mb N$. Fix a subset $V_1 \subseteq V(K_n)$ such that $|V_1|\in (n/(\log n)^{4}, n/(\log n)^{2})$. Furthermore fix $G \subseteq K_n$ such that $\Delta(G^{c})\le 2n/\log n$, and $|N(v)^{c}\cap V_1|\le 2|V_1|/\log |V_1|$ for every $v\in V(G)$, and for every $v\notin V_1$ we have $\deg_G(v) \equiv 0 \pmod 2$. 

Let $\mc H'\sim\mb{G}^{(3)}(n, (\log n)^{C_{\ref{lem:IA-iteration-simplified}}}/n)$. Then there exists an edge-disjoint triangle set $\mc H \subseteq \mc{H}'$, with $G^{\ast} \coloneqq E(\mc H)$, such that:
\begin{enumerate}
    \item $G^{\ast}[V_1]$ is stochastically dominated by sampling every edge independently with probability $(\log |V_1|)^{-20}$,
    
    \item $G\setminus G[V_1]\subseteq G^{\ast}$ (i.e., $\mc H$ covers all edges of $G$ outside of $V_1$) with probability $1-n^{-\omega(1)}$,
    
    \item $G^{\ast}\subseteq G$ (i.e., $\mc H$ consists of triangles in $G$).
\end{enumerate}

\end{lemma}

Before proving \cref{lem:IA-iteration-simplified} we show how it implies \cref{prop:IA-random-set}.

\begin{proof}[Proof of \cref{prop:IA-random-set} given \cref{lem:IA-iteration-simplified}]
Let $n = t_0 > t_1 > \cdots > t_\ell = |X|$ be a sequence of integers such that $t_{i+1} \in (t_i/(\log t_i)^{4},t_i/(\log t_i)^{2})$ for every $0 \leq i < \ell$. Observe that $\ell = O(\log n)$. Sample a uniformly random descending sequence of sets $V(K_n) = V_0 \supseteq V_1 \supseteq V_2\supseteq\cdots \supseteq V_{\ell} = X$ such that $|V_i| = t_i$ for every $i$. We call this sequence of sets the \textit{vortex}.

We now consider the respective degrees from each vertex set into the next. By applying the Chernoff bound (\cref{lem:chernoff}), a union bound, and the assumed upper bound on $|X\setminus N_G(v)|$, with probability at least $0.99$ (over the random choice of the vortex), we have that every vertex $v\in V_i$ has $|N(v)^c\cap V_{i+1}|\le 3|V_{i+1}|/(2\log |V_{i+1}|)$ and $\Delta(K_n[V_i]\setminus G[V_i])\le 3|V_i|/(2\log |V_i|)$ for $0 \leq i < \ell$. We assume these conditions hold.

We now apply \cref{lem:IA-iteration-simplified} inductively. Suppose that for some $0 \leq i < \ell$ we have already applied \cref{lem:IA-iteration-simplified} $i$ times, leaving the graph $L_i \subseteq G[V_i]$ of uncovered edges. In order to simplify the analysis we will not apply \cref{lem:IA-iteration-simplified} to $L_i$ directly. Rather, we will apply it to the graph $G_{i}^{\mr{curr}}$, defined as follows: If $i=\ell-1$ then $G_{i}^{\mr{curr}} = L_i$. Otherwise let $G_{i}^{\mr{curr}} = L_i \setminus G[V_{i+2}]$. Applying \cref{lem:IA-iteration-simplified} in this way implies that an edge in $G[V_i]$ can only be covered in the $i$-th or $(i-1)$-th stage of the algorithm (but not before).

We next note that as $|V_{i+2}|\le |V_i|/(\log |V_i|)^{5/4}$ we see that $G_i' \coloneqq G[V_i]\setminus G[V_{i+2}]$ has the property that for all $v\in V_i$, we have $|N_{G_i'}(v)^c\cap V_i|\le 7|V_i|/(4\log |V_i|)$ and $|N_{G_i'}(v)^c\cap V_{i+1}|\le 7|V_{i+1}|/(4\log |V_{i+1}|)$. Hence $G_{i}^{\mr{curr}}$ satisfies the necessary conditions for \cref{lem:IA-iteration-simplified} with high probability. Indeed, by the inductive assumption, after $i$ steps of the process $L_i$ is stochastically dominated by sampling every edge independently with probability $(\log |V_i|)^{-20}$. Thus, by Chernoff's inequality and a union bound, the two minimum degree assumptions hold w.h.p. Moreover, $L_i$ is obtained from the triangle-divisible graph $G$ by removing a set of edge-disjoint triangles. Therefore all degrees in $L_i$ are even. Since $G_i^{\mr{curr}}$ is obtained from $L_i$ by removing only edges spanned by $V_{i+2}$, the degrees of all vertices in $V_i \setminus V_{i+1}$ in $G_i^{\mr{curr}}$ are even as well. Therefore we can apply \cref{lem:IA-iteration-simplified} to $G_i^{\mr{curr}}$ and continue the process. Assuming $C_{\ref{prop:IA-random-set}}$ is sufficiently large, the failure probability in stage $i$ is less than $|V_i|^{-2}$. Thus, the total failure probability is less than
\[\sum_{i\ge 0}|V_i|^{-2}\leq \sum_{k=C_{\ref{prop:IA-random-set}}}^\infty k^{-2} \le 1/8.\]

Finally, note that in this procedure no edges in $G[X]$ are covered until the final step. Therefore the degree bound on $G[X]$ follows by noting that $G^{\ast}[X]$ is stochastically dominated by sampling every edge with probability $(\log |X|)^{-20}$.
\end{proof}

\subsection{Fractional matching}\label{sub:fractional-matching}
In order to find the existence of fractional matchings within a sparse set of triangles we will use the following result of Barber, Glock,  K\"uhn, Lo, Montgomery, and Osthus \cite{BGKLMO20}. 

\begin{lemma}\label{lem:fractional-matching}
There exists an $\varepsilon = \varepsilon_0>0$ such that the following holds. Given a graph $G$ on $n$ vertices with minimum degree at least $(1-\varepsilon)n$, let $\mc{T}$ denote the set of triangles in $G$, and, for $e \in E(G)$, let $\mc{T}(e)$ be the set of edges containing $e$. There exists a function $\gamma:\mc{T}\to [0,1]$ such that $\sum_{T\in \mc{T}(e)}\gamma(T) = n/8$ for every $e$.
\end{lemma}
\begin{remark}
This follows from \cite[Lemma~4.2]{BGKLMO20}, noting that any sufficiently dense graph is regular in the appropriate sense and letting $\gamma(\cdot) = \psi(\cdot) \cdot 1/(2p^2)$ where $\psi(\cdot)$ is defined as in \cite[Lemma~4.2]{BGKLMO20} and $p$ is the density of $G$.
\end{remark}

The crucial tool for our setting is that given such a fractional matching, one can subsample every triangle with weight proportional to $\gamma$ and obtain a nearly perfect fractional matching inside the sampled hypergraph.
\begin{lemma}\label{lem:fractional-matching-sample}
There exists an $\varepsilon = \varepsilon_0>0$ such that the following holds. Given a graph $G$ on $n$ vertices with minimum degree at least $(1-\varepsilon)n$ let $\mc{T}$ denote the set of triangles in $G$. Sample every triangle in $G$ with probability $p$ with $p\ge (\log n)^2/n$ and call this collection $\mc{H}$. Then with probability $1-n^{-\omega(1)}$, there exists a triangle set $\mc{H}_1\subseteq \mc{H}$ such that every edge is contained in $pn/8\pm\sqrt{pn}\log n$ triangles of $\mc{H}_1$.
\end{lemma}
\begin{proof}
Let $\gamma(\cdot)$ be as in \cref{lem:fractional-matching}. Let $\mc{H}_1$ be the random model where every triangle is sampled with probability $\gamma(T)\cdot p$ and note we can couple $\mc{H}_1\subseteq\mc{H}$. The result then follows immediately from the Chernoff bound, noting that the expected number of triangles containing a given edge $e$ is $\sum_{T\in\mc{T}(e)}p\gamma(T) = pn/8$. 
\end{proof}

\subsection{Covering process within regular triangle subset}\label{sub:nibble}
We now show that we can cover most of the edges of an almost-complete graph using a sparse random triangle set. We will first require a set of notions with regards to hypergraph matchings. For a hypergraph $\mc{H}$, define
\[\Delta(\mc{H}) :=\max_{v\in V(\mc{H})}\deg_\mc{H}(v),~\Delta^{\mr{co}}(\mc{H}) :=\max_{v_1,v_2\in V(\mc{H})} \on{codeg}_{\mc{H}}(v_1,v_2).\]
Call a function $\omega:E(\mc{H})\to \mb{R}_{\ge 0}$ a \textit{weight function}, and for $X \subseteq E(\mc{H})$ let $\omega(X) = \sum_{x\in X}\omega(x)$. We will require the following result of Ehard, Glock, and Joos \cite{EGJ20} which guarantees the existence of hypergraph matchings which are pseudorandom with respect to a collection of weight functions.

\begin{theorem}[{\cite[Theorem~1.2]{EGJ20}}]\label{thm:matching}
Suppose $\delta\in(0,1)$ and $r\in \mb{N}$ with $r\ge 2$, and let $\eps:=\delta/(50r^2)$. Then there exists $\Delta_0$ such that for all $\Delta\ge \Delta_0$ the following holds:
Let $\mc{H}$ be an $r$-uniform hypergraph with $\Delta(\mc{H})\leq \Delta$ and $\Delta^{\mr{co}}(\mc{H})\le \Delta^{1-\delta}$ as well as $e(\mc{H})\leq \exp(\Delta^{\eps^2})$. 
Suppose that $\mc{W}$ is a set of at most $\exp(\Delta^{\eps^2})$ weight functions on~$E(\mc{H})$.
Then, there exists a matching $\mc{M}$ in~$\mc{H}$ such that $\omega(\mc{M})=(1\pm \Delta^{-\eps}) \omega(E(\mc{H}))/\Delta$ for all $\omega \in \mc{W}$ with $\omega(E(\mc{H}))\ge \max_{e\in E(\mc{H})}\omega(e)\Delta^{1+\delta}$.
\end{theorem}

This immediately implies the following lemma. We include the proof for completeness.

\begin{lemma}\label{lem:nibble-cover}
There exists $\varepsilon = \varepsilon_0>0$, $C = C_{\ref{lem:nibble-cover}}>0$ such that the following holds for sufficiently large $n$. Fix $p\in ((\log n)^{C}/n,1)$, a graph $G$ on $n$ vertices with minimum degree at least $(1-\varepsilon_0)n$, and let $\mc{H}_1$ be a collection of triangles in $G$ such that each edge is in $pn/8\pm\sqrt{pn}\log n$ triangles. Then there exists a set of edge disjoint triangles $\mc{H}_2 \subseteq \mc{H}_1$ such that $\Delta(G\setminus E(\mc{H}_2))\le n/(\log n)^{1000}$. 
\end{lemma}

\begin{proof}
Define the auxiliary $3$-uniform hypergraph $\mc{H}$ with vertices corresponding to the set of edges in $G$ and $3$-edges corresponding to the triangles in $\mc{H}_1$. Notice that $\Delta(\mc{H}) = pn/8 + O(\sqrt{pn}\log n)$. Furthermore as any pair of edges are contained in at most $1$ triangle, it follows that $\Delta^{\mr{co}}(\mc{H})\le 1$. Thus \cref{thm:matching} applies with $\delta = 1/2$ and therefore $\varepsilon = 1/(900)$, and $\Delta = \Delta(\mc{H})$. 

We now define the weight functions. For a vertex $v$, let $w_v$ be $1$ on all $3$-edges of $\mc{H}$ corresponding to triangles containing $v$, and $0$ elsewhere. Note that $w_v(E(\mc{H})) \ge pn^2/32\ge \Delta(\mc{H})^{1+\delta}$. Furthermore if $C$ is sufficiently large we have that $\exp(\Delta^{\varepsilon^2})\ge n$ and thus there is a matching $\mc{M}$ in $\mc{H}$ such that
\[
w_v(\mc{M})\ge \frac{(1 \pm \Delta^{-\varepsilon})}{\Delta} \cdot \frac{(\Delta - O(\sqrt{pn}\log n))\deg_G(v)}{2}\ge \frac{(1- 2\Delta^{-\varepsilon})\deg_G(v)}{2}
\]
for all $v\in V(G)$. This implies that the matching, which corresponds to triangles of $G$, covers all but $2\Delta^{-\varepsilon}n$ edges incident to $v$. Taking $C$ sufficiently large the result follows immediately.
\end{proof}

In the next three sections we prove \cref{lem:IA-iteration-simplified}. Our proof closely follows the proof of  \cite[Lemma 3.8]{BGKLMO20}, with the necessary adaptations to account for the random triangle set $\mc{H}'$ and the fact that $|V_1| / |V(G)|$ is relatively smaller in our setting than in \cite{BGKLMO20}.

\subsection{Setup for iterative absorption}\label{sub:setup}
We are now in position to apply the results of \cref{lem:fractional-matching-sample} and \cref{lem:nibble-cover}. However, we cannot simply invoke these results on the whole graph $G$; a more delicate approach is required.

Recall that we have a graph $G \subseteq K_n$ with a distinguished vertex subset $V_1$. Let $q = (\log |V_1|)^{-30}$. Let $R$ be a set of edges in $G[V(K_n)\setminus V_1, V_1]$ with the following properties:
\begin{enumerate}[({\bfseries{A\arabic{enumi}}})]
    \item For all $v\in V(G)\setminus V_1$, $\deg_{R} v = q|V_1| + O(q|V_1|(\log |V_1|)^{-1})$.
    \item For all $v\in V_1$, $\deg_{R}v = qn + O(qn(\log n)^{-1})$.
    \item For all $v\in V(G)\setminus V_1$, $v'\in V_1$, we have that $|N_R(v)\cap N_G(v')| = q|V_1| + O(q|V_1|(\log |V_1|)^{-1})$.
    \item For all $v,v'\in V(G)\setminus V_1$, we have that $|N_R(v)\cap N_R(v')|\ge q^2|V_1|/2$.
    \item For all $v,v'\in V_1$ we have that $|N_R(v)\cap N_R(v')|\le 2q^2n$.
\end{enumerate}

That such a graph $R$ exists is established by noting that if each edge in $G[V(K_n)\setminus V_1, V_1]$ is independently sampled with probability $q$ then, by Chernoff's inequality and a union bound, these properties hold with positive probability.

Let $G_1 = G\setminus (R\cup G[V_1])$. It is easy to see that $\delta(G_1) \ge |V(G)|-O(|V(G)|/(\log |V(G)|)$. Let $\mc{T}$ denote the set of triangles in $G_1$ and sample each triangle in $\mc{T}$ with probability $p = (\log n)^{2C_{\ref{lem:nibble-cover}}}\!/n$ to form the random set $\mc{H}''$. By \cref{lem:fractional-matching-sample}, with probability $1-n^{-\omega(1)}$ there exists a subset of triangles $\mc{H}_1\subseteq \mc{H}''$ such that every edge of $G_1$ is in $pn/8\pm\sqrt{pn}\log n$ triangles. Applying \cref{lem:nibble-cover}, we find that there exists a set of edge-disjoint triangles in $\mc{H}_1$ that covers all edges of $G_1$ except at most $n/(\log n)^{100}$ incident to each vertex. Let $L \subseteq G_1$ be the graph of uncovered edges. Let $L_1 \subseteq L$ be the ``internal'' edges with no vertex in $V_1$ and let $L_2 \coloneqq L \setminus L_1$ be the uncovered ``crossing'' edges with an endpoint in $V_1$. We remark that since $G_1$ contains no edges with both vertices in $V_1$, neither does $L$.

It remains to cover $G_2 \coloneqq L_1 \sqcup L_2 \sqcup R$ with triangles while not covering too many edges in $G[V_1]$. Let $R_2 \coloneqq L_2 \cup R$. Observe that $R_2$ satisfies:
\begin{enumerate}[({\bfseries{B\arabic{enumi}}})]
    \item For all $v\in V(G)\setminus V_1$, we have that $\deg_{R_2}(v) = q|V_1| + O(q|V_1|(\log |V_1|)^{-1})$.
    \item For all $v\in V_1$, we have that $\deg_{R_2}(v) = qn + O(qn/\log n)$.
    \item For all $v\in V(G)\setminus V_1$, $v'\in V_1$, we have that $|N_{R_2}(v)\cap N_G(v')| = q|V_1| + O(q|V_1|(\log |V_1|)^{-1})$.
    \item For all $v,v'\in V(G)\setminus V_1$, we have that $|N_{R_2}(v)\cap N_{R_2}(v')|\ge q^2|V_1|/2$.
    \item For all $v,v'\in V_1$ we have that $|N_{R_2}(v)\cap N_{R_2}(v')|\le 3q^2n$.
\end{enumerate}

We complete the construction by first covering the internal edges that comprise $L_1$ and then covering the remaining crossing edges.

\subsection{Cover-down stage \texorpdfstring{$1$}{1}: internal edges}\label{sub:cover-down-1}

\begin{lemma}\label{lem:cover-down-1}
With the above setup, let $\mc{T}_2$ denote the set of triangles in $G_2$ and let  $\mc{H}_3 \subseteq \mc{T}_2$ be a random set of triangles with each triangle included with probability $(\log n)^{100}/n$. Then with probability $1-n^{-\omega(1)}$, one can choose edge disjoint triangles $\mc{H}_4\subseteq \mc{H}_3$ such that $L_1 \subseteq E(\mc{H}_4)$.
\end{lemma}

\begin{proof}
We construct $\mc{H}_4$ with a random greedy algorithm. Order the edges in $L_1$ arbitrarily. When processing an edge $e$, expose the triangles of $\mc{H}_3$ containing $e$ and then choose one such triangle, not overlapping with previous choices, uniformly at random and add it to $\mc{H}_4$. This procedure only fails if for some $e$, all triangles containing it in $\mc{H}_3$ overlap previous choices. However note that initially each edge in $L_1$ is contained in at least $q^2|V_1|/2$ triangles in $\mc{T}_2$, and that triangles added previously to $\mc{H}_4$ eliminate at most $2n/(\log n)^{100}\le q^2|V_1|/4$ of these. Thus the expected number of extensions at each stage is at least $q^2|V_1|/4 \cdot (\log n)^{100}/n\gtrsim (\log n)^{50}$. Applying Chernoff's inequality and a union bound, with probability $1-n^{-\omega(1)}$ there is at least one choice for each stage.
\end{proof}

Given \cref{lem:cover-down-1}, the remaining graph to cover is $R_3 \coloneqq R_2\setminus E(\mc{H}_4)$. We note that since every triangle in $\mc{H}_4$ involves an edge of $L_1$, we have:
\begin{enumerate}[({\bfseries{C\arabic{enumi}}})]
    \item\label{C1} For all $v\in V(G)\setminus V_1$, we have that $\deg_{R_3}(v) = q|V_1| + O(q|V_1|/\log |V_1|)$.
    
    \item\label{C2} For all $v\in V_1$, we have that $\deg_{R_3}(v) = qn + O(qn/\log n)$.
    
    \item\label{C3} For all $v\in V(G)\setminus V_1$, $v'\in V_1$, we have that $|N_{R_3}(v)\cap N_G(v')| = q|V_1| + O(q|V_1|/\log |V_1|)$.
    
    \item\label{C4} For all $v,v'\in V_1$, we have that $|N_{R_3}(v)\cap N_{R_3}(v')|\le 3q^2n$.
\end{enumerate}

\subsection{Cover-down stage \texorpdfstring{$2$}{2}: crossing edges}\label{sub:cover-down-2}

Our goal is to cover $R_3$ using only a small number of edges from $G[V_1]$. This will be accomplished by reducing the problem to a simultaneous matching problem on link graphs of vertices in $V(G) \setminus V_1$.

We first require the following lemma. It is an immediate consequence of the (substantially stronger) main results in~\cite{KLS14, Joh20}. We include an elementary proof for completeness.

\begin{lemma}\label{lem:graph-subsample}
Let $G'$ be a graph on $N$ vertices, with $N$ even, and with minimum degree at least $3N/4$. Let $H$ be a random subgraph of $G'$ where each edge is sampled independently with probability $(\log N)^2/N$. Then $H$ has a perfect matching with probability $1-N^{-\omega(1)}$. 
\end{lemma}

The proof of \cref{lem:graph-subsample} uses the following convenient Hall-type criterion for a bipartite graph to have a perfect matching. It is an immediate consequence of the main theorem in \cite{SS17}.

\begin{lemma}\label{lem:hall}
Let $G' = (X\cup Y, E)$ be a bipartite graph with $|X| = |Y| = N$. Suppose that for every $S\subseteq X$, $S'\subseteq Y$ with $|S'|< |S|\le\lceil N/2\rceil$ we have $e(S,Y\setminus S')\neq 0$, and that for every $T'\subseteq X$, $T\subseteq Y$ with $|T'|<|T|\le\lceil N/2\rceil$ we have $e(T,X\setminus T')\neq 0$. Then $G$ has a perfect matching.
\end{lemma}

\begin{proof}[{Proof of \cref{lem:graph-subsample}}]
Consider a uniformly random equipartition $X\cup Y$ of $V(G')$ and let $G^\dagger \coloneqq G'[X,Y]$. By the Chernoff bound for hypergeometric random variables and a union bound with probability $1 - N^{-\omega(1)}$ we have $\deg_{G^\dagger}(v) \ge N/3$ for each vertex $v$. Now consider some $S\subseteq X, S'\subseteq Y$ satisfying $\lceil N/4\rceil\ge|S|>|S'|$. It holds that
\begin{align*}
e_{G^\dagger}(S,Y\setminus S') &= \sum_{v\in S}\deg_{G^\dagger}(v) - e_{G'}(S,S')\ge (N/3)|S| - |S|^2\ge N|S|/12.
\end{align*}
Similarly, if $T'\subseteq X, T\subseteq Y$ with $|T'|<|T|\le\lceil N/4\rceil$ then $e_{G^\dagger}(T,X\setminus T')\ge N|S|/12$.

We observe that with probability at least $1-\exp(-\Omega((\log N)^2))$, we have $e_{H}(S,Y\setminus S')>0$ for every pair of sets $S\subseteq X, S'\subseteq Y$ with $\lceil N/4\rceil\ge|S|>|S'|$.  Indeed, by a union bound over $S,S'$ and the Chernoff bound, the probability that this fails to hold is at most
\begin{align*}
	\sum_{k=1}^{\lceil N/4\rceil}\binom{N/2}{k}\sum_{\ell=0}^{k-1} \binom{N/2}{\ell}\exp(-\Omega(k(\log N)^2)) & \le\sum_{k=1}^{\lceil N/4\rceil} N^{2k}\exp(-\Omega(k(\log N)^2)))\\
	& \le\exp(-\Omega((\log N)^2)).
\end{align*}
By symmetry, the same is true (with probability at least $1-\exp(-\Omega((\log N)^2))$) when switching the roles of $X$ and $Y$. The desired result then follows from \cref{lem:hall}.
\end{proof}

\begin{lemma}\label{lem:run-match}
Fix $R_3$ as in \cref{sub:cover-down-1} and suppose that it satisfies \cref{C1}-\cref{C4}. Let $\mc{T}_3$ be the set of triangles in $R_3\cup G[V_1]$ and let $\mc{H}_5 \subseteq \mc{T}_3$ be a random set of triangles with each triangle included with probability $(\log |V_1|)^{2}/(q|V_1|)$. We then have:
\begin{itemize}
    \item For each edge in $G[V_1]$, the probability that it is contained in a triangle in $\mc{H}_5$ is at most $(\log n)^{-25}$. Moreover, these events are mutually independent.
    
    \item With probability $1-n^{-\omega(1)}$, there exists a set of edge disjoint triangles $\mc{H}_6\subseteq \mc{H}_5$ such that $R_3 \subseteq E(\mc{H}_6)$.
\end{itemize}
\end{lemma}
\begin{proof}
For the first point, recall that by \cref{C4} every pair of distinct $u,v \in V_1$ has at most $3q^2n$ common neighbors in $R_3$. Thus, the probability for an edge to be contained in a triangle in $\mc{H}_5$ is at most $3q^2n \cdot (\log |V_1|)^{2}/(q|V_1|)\le (\log n)^{-25}$. Moreover, for distinct edges in $G[V_1]$, their extensions into triangles in $\mc{H}_5$ are disjoint, implying mutual independence of these events.

For the second point, order the vertices in $V(G)\setminus V_1$ and consider them sequentially. Suppose we are processing $v$ and note first that by \cref{C1} the graph spanned by the triangles containing $v$ in $R_3\cup G[V_1]$ has $q|V_1| + O(q|V_1|(\log |V_1|)^{-1})$ vertices. Furthermore by \cref{C3} its link graph (i.e., the subgraph spanned by $V_1$)has minimum degree $q|V_1| + O(q|V_1|(\log |V_1|)^{-1})$. Our goal is to find a perfect matching in this link graph that is edge-disjoint from previously found perfect matchings. These perfect matchings correspond to the desired $\mc{H}_6$.

At every step the set of edges removed is stochastically dominated by our random sample of edges at rate $(\log n)^{-25}$, so we find that the minimum degree in this link is essentially unchanged by the previous triangles removed in this process. Therefore if one samples each triangle in the link with probability $(\log|V_1|)^{2}/(q|V_1|)$, or equivalently each edge in the link with probability $(\log|V_1|)^{2}/(q|V_1|)$, by \cref{lem:graph-subsample} we can construct a matching for $v$ with probability $1-n^{-\omega(1)}$. This immediately gives the desired result.
\end{proof}

We are now in position to prove \cref{lem:IA-iteration-simplified}.

\begin{proof}[Proof of \cref{lem:IA-iteration-simplified}]
Using the construction above, let $\mc{H}^\ast \coloneqq \mc{H}_2\cup \mc{H}_4\cup \mc{H}_6$. Observe that $\mc{H}^\ast$ is stochastically dominated by a random hypergraph of the specified density. Furthermore, if $R_3$ is suitable in the sense of \cref{sub:cover-down-1} then $E(\mc{H}^\ast)[V_1]$ is stochastically dominated by a random hypergraph of the appropriate density as well. Hence, we may take $\mc{H} = \mc{H}^\ast$ if $R_3$ is suitable. Otherwise we take $\mc{H} = \emptyset$. Since $R_3$ is suitable with probability $1 - n^{-\omega(1)}$ and $\mc{H}^\ast$ covers all edges in $G \setminus G[V_1]$ with probability $1 - n^{-\omega(1)}$, the result follows.
\end{proof}

\section{Modifications for Latin squares}\label{sec:latin-squares}
In this section we briefly discuss the necessary changes to prove \cref{thm:latin}, as opposed to \cref{thm:main}. Since these changes are largely superficial, we do not repeat the arguments in detail.

A Latin square can be thought of as a triangle-decomposition of $K_{n,n,n}$, with vertex parts $V^1,V^2,V^3$, each of size $n$. We say that a tripartite subgraph of $K_{n,n,n}$ is \emph{triangle-divisible} if for every $j\in[3]$ and vertex $v\in V^j$, its degrees into $V^{j-1}$ and $V^{j+1}$ are the same (taking indices $\bmod{3}$). The analogue of \cref{thm:boosting} is the following:
\begin{theorem}\label{thm:latin-boosting}
Fix a triangle-divisible tripartite graph $G\subseteq K_{n,n,n}$ with $\Delta(K_{n,n,n}\setminus G)\le n/\log n$ and $n\ge \exp(C_{\ref{thm:latin-boosting}}/\eta^4)$. Let $\mc{H}$ be the result of randomly sampling each triangle of $K_{n,n,n}$ with probability $n^\eta/n$. With probability at least $1/2$ the collection $\mc{H}$ contains a triangle-decomposition of $G$.
\end{theorem}

The proof strategy is similar, and we detail the necessary changes.
\begin{itemize}
	\item To prove an analogue of \cref{prop:IA-random-set}, the vortex $V(K_{N,N,N})=V_0\supseteq\cdots\supseteq V_\ell = X$ should be chosen so that each $V_k$ has the same number of vertices in each $V^j$.
	
	\item During the iteration, replace the various degree typicality assumptions (e.g.~\cref{C1,C2,C3,C4}) with the obvious tripartite analogues.
	
	\item In the final step of the cover-down procedure (\cref{sub:cover-down-2}), in the original setup we reduced to a bipartite matching problem by taking a random bipartition of $U_{i+1}$. In the Latin square setting this is not necessary since the bipartite structure is already induced by $K_{n,n,n}$.

	\item The existence of the regular triangle subset (\cref{lem:fractional-matching}) relies on weight-shifting gadgets which are not tripartite. It is possible to adapt work of Montgomery \cite{Mon17} to obtain a suitable approximate tripartite fractional matching result; see e.g.~\cite[Lemma~8.11]{KSSS22b}.
	
	\item The absorbing structures used in \cref{Alg6,Alg7} within the proof of \cref{thm:boosting} must be tripartite. However, this is not an obstruction since the vertices of $\mc{F}_{2m}$ can be split into three classes so that all hyperedges are tripartite.
\end{itemize}

\bibliographystyle{amsplain0.bst}
\bibliography{main.bib}

\end{document}